\documentclass{article}
\usepackage{amsthm,amsmath,amsfonts,amssymb}

\title{Relative entropy and the stability of shocks and contact discontinuities for systems of conservation laws with non $BV$ perturbations}

\author{Nicholas {\sc Leger} and Alexis {\sc Vasseur} \\ \small Department of Mathematics \\ \small University of Texas at Austin}

\newlength{\hchng}
\newlength{\vchng}
\newtheorem{theo}{Theorem}[section]
\newtheorem{prop}{Proposition}
\newtheorem{corr}{Corollary}
\newtheorem{lemm}{Lemma}
\newtheorem{defi}{Definition}

\newcommand{\ds}{\displaystyle}
\newcommand{\R}{\mathbb R}
\newcommand{\eps}{\varepsilon}

\newcommand{\dt}{\partial_t}
\newcommand{\dx}{\partial_x}

\newcommand{\mV}{\mathcal{V}}
\newcommand{\mI}{\mathcal{I}}
\newcommand{\mU}{\mathcal{U}}
\newcommand{\Vs}{V_{\mathrm{max}}}
\newcommand{\Vi}{V_{\mathrm{min}}}
\newcommand{\sU}{{U_{\pm}}}
\newcommand{\midd}{\thinspace\vert\thinspace}
\DeclareMathOperator*{\esssup}{ess\,sup}

\begin{document}
\maketitle
\bibliographystyle{plain}

\noindent{\bf Abstract:}  We develop a theory based on relative entropy to show the uniqueness and $L^2$ stability (up to a translation) of extremal entropic Rankine-Hugoniot discontinuities for systems of conservation laws (typically 1-shocks, n-shocks, 1-contact discontinuities and n-contact discontinuities of large amplitude) among bounded entropic weak solutions having an additional trace property. The existence of a convex entropy is needed. 
No $BV$ estimate is needed on the weak solutions considered. 
The theory holds without smallness condition.
The assumptions are quite general. For instance, strict hyperbolicity is not needed globally. For fluid mechanics, the theory handles solutions with vacuum.   

\vskip0.3cm \noindent {\bf Keywords:}
System of conservation laws, compressible Euler equation, Rankine-Hugoniot discontinuity, shock, contact discontinuity, relative entropy, stability, uniqueness. 

\vskip0.3cm \noindent {\bf Mathematics Subject Classification:}
35L65, 35L67, 35B35.

\section{Introduction}

In this article, we develop a theory for uniqueness and global $L^2$ stability of extremal entropy-admissible Rankine-Hugoniot discontinuities (typically 1-shocks, n-shocks, 1-contact discontinuities and n-contact discontinuities)
for a wide class of systems of conservation laws endowed with a convex entropy. The uniqueness and stability is shown in the class of bounded weak entropic solutions verifying the following trace property.
\begin{defi}\label{defi_trace}
Let $U\in L^\infty(\R^+\times\R)$. We say that $U$ verifies the strong trace property if for any Lipschitzian curve $t\to X(t)$, there exists two bounded functions $U_-,U_+\in L^\infty(\R^+)$ such that for any $T>0$
$$
\lim_{n\to\infty}\int_0^T\sup_{y\in(0,1/n)}|U(t,x(t)+y)-U_+(t)|\,dt=\lim_{n\to\infty}\int_0^T\sup_{y\in(-1/n,0)}|U(t,x(t)+y)-U_-(t)|\,dt=0.
$$
\end{defi}
Note that, for each fixed curve,  this is equivalent to the convergence for almost every time $t$.
Obviously, any $BV$ function verifies this strong trace property.
But this requirement is weaker than the $BV$ property.
Let us emphasize that this notion of trace is more restrictive than the strong trace introduced in \cite{Vasseur_trace}, which is 
 known to be verified for bounded solutions of scalar conservation laws. 
This has been shown in the multidimensional case, first with a non-degeneracy property, in \cite{Vasseur_trace}. In the one-dimensional case, a different proof based on compensated compactness was proposed by  Chen and Rascle \cite{Chen_trace}.
For a general flux function the strong trace problem has been solved in the 1D case in \cite{KV}. The general multidimensional case has been obtained by Panov \cite{Panov,Panov2} (see also Kwon \cite{Kwon},  De Lellis,  Otto,  and Westdickenberg \cite{DeLellis} for interesting generalizations). In the case of systems, this is mainly an open problem. This has been shown only for the particular case of isentropic gas dynamics with $\gamma=3$ for traces in time (traces in space can be shown the same way) in \cite{Vasseur_gamma3}. 
Unfortunately, there are no such results for the strong  trace property of Definition \ref{defi_trace} outside of the usual $BV$ theory.
\vskip0.3cm
Stability of shocks in the class of $BV$ solutions has been investigated by a number of authors. In the case of small perturbations in $L^\infty\cap BV$, Bressan, Crasta, and Piccoli \cite{Bressan1} developed a powerful theory of $L^1$ stability for entropy solutions obtained by either the Glimm scheme \cite{Glimm} or the wave front-tracking method.  A simplified approach has been proposed by Bressan, Liu, and Yang \cite{Bressan2} and Liu and Yang \cite{liu}. (See also \cite{Bressan}.) The theory also works in some cases for small perturbations in $L^\infty\cap BV$ of large shocks. See, for instance, Lewicka and Trivisa \cite{Lewicka} or Bressan and Colombo \cite{Bressan3}.  
 \vskip0.3cm
However, our stability result goes beyond the known results valid  in the class of $BV$ solutions, with perturbations small in $BV$. 
Our approach is based on the relative entropy method first used by Dafermos and DiPerna to show $L^2$ stability and uniqueness of Lipschitzian solutions to conservation laws
\cite{Dafermos4,Dafermos1,DiPerna}. Note that in \cite{DiPerna}, uniqueness of small shocks for strictly hyperbolic $2\times 2$ systems is shown in a class of admissible weak solutions with small oscillation in $L^\infty \cap BV$. The analysis in \cite{DiPerna} also implies the uniqueness of shocks for $2 \times 2$ systems in the Smoller-Johnson class \cite{Smoller}. In each case genuine nonlinearity is assumed. The ideas of DiPerna were developed further by Chen and Frid in the papers \cite{Frid2,Frid1}. In subsequent work, they established, together with Li \cite{Chen1}, 
the uniqueness of solutions to the Riemann problem in a large class of entropy solutions (locally $BV$ without smallness conditions) for the $3\times 3$ Euler system in Lagrangian coordinates. They also establish a large-time stability result in this context. See also Chen and Li \cite{Chen_Li} for an extension to the relativistic Euler equations. However, no stability in $L^2$ for all time is included in those results. 
\vskip0.3cm 
Our approach is based on fairly mild assumptions on the system and the Rankine-Hugoniot discontinuity.
Basically, we need the discontinuity to be extremal (1-shock or n-shock and well separated from the other Hugoniot discontinuities), to be either a contact discontinuity or verify the Liu condition.
In the case of a Liu shock, the corresponding shock curve should satisfy the Liu property everywhere (so that the shock speed varies monotonically along the curve), and we need also a property of growth of the strength of the shock along the shock curve, where the strength is measured via the entropy. Very little constraint is needed on the other shock families. Lax properties are typically enough. But we may even relax it
to cases where the system is neither genuinely nonlinear nor strictly hyperbolic, and even to cases where the shock curves are not well-defined. The theory works fine even for large shocks.  Note that the present study is another step in the program described in \cite{Vasseur_Book}. A first step was achieved for scalar conservation laws in \cite{Leger}.
\vskip0.3cm 
 
We will give a precise description of our hypotheses and main results in the next section. First let us mention a few particular cases in which our theory applies. Our first examples include the isentropic Euler system and the full Euler system for a polytropic gas. Both systems are treated in Eulerian coordinates. The isentropic Euler system is the following.
\begin{equation}\label{isentropic}
\left\{\begin{array}{l}
\ds{\dt \rho+\dx (\rho u)=0}\\[0.2cm]
\ds{\dt (\rho u )+\dx (\rho u^2+ P(\rho))=0}.
\end{array}\right.
\end{equation}
We assume a smooth pressure law $P: \R^+ \to \R$ with the following properties
\begin{align}\label{isentropic2}
P^{\prime}(\rho)>0, \qquad \quad [\rho P(\rho)]^{\prime \prime} \geq 0.
\end{align}
As usual, we consider only entropic solutions of this system, namely, those verifying additionally the entropy inequality:
$$
\dt \eta(\rho,\rho u)+\dx G(\rho, \rho u) \leq 0,
$$
with
 $$
 \eta(\rho, \rho u)=\frac{(\rho u)^2}{2 \rho} + S(\rho), \qquad G(\rho, \rho u) = \frac{(\rho u)^3}{2 \rho^2} + \rho u \thinspace S^{\prime}(\rho),
 $$
and with $S^{\prime \prime}(\rho) = \rho^{-1} P^{\prime}(\rho) > 0$. Note that we need only a single convex entropy, even if in this case there exists an entire family of convex entropies. 
\vskip0.3cm

The full Euler system reads
\begin{equation}\label{Euler}
\left\{\begin{array}{l}
\ds{\dt \rho+\dx (\rho u)=0}\\[0.3cm]
\ds{\dt (\rho u )+\dx (\rho u^2+ P)=0}\\[0.3cm]
\ds{\dt (\rho E )+\dx (\rho u E + uP)=0,}
\end{array}\right.
\end{equation}
where $E = \frac{1}{2}u^2 + e$. The equation of state for a polytropic gas is given by 
\begin{equation}\label{polytropic}
P = (\gamma-1) \rho e
\end{equation}
where $\gamma >1$. 
In that case, we consider the entropy/entropy-flux pair
\begin{equation}\label{Eulerconvex}
\eta(\rho, \rho u, \rho E) = (\gamma -1) \rho \ln \rho - \rho \ln e, \qquad G(\rho, \rho u, \rho E) = (\gamma -1) \rho u \ln \rho - \rho u \ln e,
\end{equation}
where, in conservative variables, we have $e = \ds{\frac{\rho E}{\rho} - \frac{(\rho u)^2}{2 \rho^2}}$. \\

For the Euler systems (\ref{isentropic}) and (\ref{Euler}), we have the following theorem.

\begin{theo}\label{theoEuler}
Consider a  shock $(U_L,U_R)=((\rho_L,u_L),(\rho_R,u_R))$ with velocity $\sigma$ associated  to the system (\ref{isentropic})-(\ref{isentropic2}), (resp. $(U_L,U_R)=((\rho_L,u_L,E_L),(\rho_R,u_R,E_R))$ associated  to the system (\ref{Euler})-(\ref{polytropic})). For any $K>0$, there exists $C_K>0$ with the following property. For any $0<\eps<1$ we have the two following cases:
\begin{itemize}
\item If $(U_L,U_R)$ is a 1-shock, then for any weak entropic solution $U=(\rho,u)\in L^\infty(\R^+ \times\R)$ of (\ref{isentropic}) (resp.  $U=(\rho, u,E)\in L^\infty(\R^+ \times \R)$ of (\ref{Euler})) 
verifying the strong trace property of Definition \ref{defi_trace} and such that $\|(\rho, u)\|_{L^\infty}\leq K$, (resp. $\|(\rho, u,E)\|_{L^\infty}\leq K$) and
\begin{eqnarray*}
\int_0^\infty |U_0(x)- U_R|^2\,dx\leq \eps,\qquad  
\int_{-\infty}^0|U_0(x)-U_L|^2\,dx\leq \eps^4,
\end{eqnarray*}
then there exists a Lipschtiz path $x(t)$ such that for any $T>0$ we have both
\begin{eqnarray*}
\int_0^\infty |U(T,x+x(T))- U_R|^2\,dx\leq C\eps(1+T),\qquad
\int_{-\infty}^0|U(T,x+x(T))- U_L|^2\,dx\leq \eps^4.
\end{eqnarray*}
\item If $(U_L,U_R)$ is a n-shock, then for any weak entropic solution $U\in L^\infty(\R^+\times \R)$ of (\ref{isentropic}) (resp. of (\ref{Euler})) 
verifying the strong trace property of Definition \ref{defi_trace} and such that $\|(\rho,u)\|_{L^\infty}\leq K$ (resp. $\|(\rho,u,E)\|_{L^\infty}\leq K$), and
\begin{eqnarray*}
\int_0^\infty |U_0(x)- U_R|^2\,dx\leq \eps^4,\qquad 
\int_{-\infty}^0|U_0(x)- U_L|^2\,dx\leq \eps,
\end{eqnarray*}
then there exists a Lipschtiz path $x(t)$ such that for any $T>0$ we have
\begin{eqnarray*}
\int_0^\infty |U(T,x+x(T)- U_R|^2\,dx\leq \eps^4,\qquad
\int_{-\infty}^0|U(T,x+x(T))- U_L|^2\,dx\leq C\eps(1+T).
\end{eqnarray*}
\end{itemize}
In both case we have 
$$
|x(t)-\sigma t|\leq C\sqrt{\eps t(1+t)}.
$$
\end{theo}
Note that this implies both the uniqueness and the stability of the shocks. The solutions are not assumed to be away from vacuum. This gives a $L^2$ stability results up to the translation $x(t)$. It is worth noting that the treatment of the entropy on the left and the right of the shock is not the same.
For example, in the case of a 1-shock, we will show that the total relative entropy on the left is strictly decreasing. Hence, if at $t=0$, $U^0(x)=U_L$ for $x<0$, then it stays that way on the left 
of $x(t)$ (while keeping the relative entropy on the right under control). The idea here comes from \cite{Vasseur_shock}, where a similar stability appeared in the study of a semi-discrete shock for an isentropic gas with $\gamma=3$. For $2 \times 2$ systems, all Rankine-Hugoniot discontinuities are extremal, hence Theorem \ref{theoEuler} applies to all admissible shocks of (\ref{isentropic}). (The theorem also holds for contact discontinuities in this case, although the pressure laws are usually nonphysical.) For the full Euler system, 
all shocks are 1-shocks or n-shocks (3-shocks in this case), so again the result applies to any entropy admissible shock. However, our result does not provide the stability of contact discontinuities for this system. The problem is that the contact discontinuities correspond to 2-waves (with a middle eigenvalue). Note that in the isentropic case with $P(\rho) = \rho^{\gamma}$ ($\gamma > 1$), it is enough to assume that the initial values are bounded since solutions can be constructed conserving this property (see Chen \cite{Chen3}, or Lions Perthame Tadmor \cite{LPT}, for instance). 
\vskip0.3cm
We now show an application of our method in the general setting of strictly hyperbolic conservation laws with either linearly degenerate or genuinely nonlinear characteristic fields.
We consider an $n \times n$ system of conservation laws
\begin{align}\label{cl1}
\dt U+\dx A(U)=0,
\end{align}
which has a strictly convex entropy $\eta$. Assume that $A$ and $\eta$ are of class $C^2$ on an open state domain $\Omega \subset \R^n$.
We have the following result.

\begin{theo}\label{theoHyper}
Assume that the smallest (resp. largest) eigenvalue of $\nabla A(V)$ is simple for all $V \in \Omega$, and that the corresponding 1-characteristic family (resp. n-characteristic family) of (\ref{cl1}) is either genuinely nonlinear or linearly degenerate. Then, for any $V_0 \in \Omega$, there exists $K>0$ and $C>0$ such that, for any entropy-admissible 1-shock or 1-contact discontinuity (resp. n-shock or n-contact discontinuity) with speed $\sigma$ and endstates $(U_L, U_R)$ verifying $U_L \in B_K(V_0)$ and $U_R \in B_K(V_0)$, the following is true.
For $\eps>0$ small enough, and for any weak entropic solution $U$ bounded in $B_K(V_0)$ on $(0,T)$ such that
\begin{eqnarray*}
\int_0^\infty |U_0(x)-U_R|^2\,dx\leq \eps,\qquad
\int_{-\infty}^0 |U_0(x)-U_L|^2\,dx\leq \eps^4,
\end{eqnarray*}
(resp. $\int_0^\infty |U_0(x)-U_R|^2\,dx\leq \eps^4$ and $\int_{-\infty}^0 |U_0(x)-U_L|^2\,dx\leq \eps$), there exists a Lipschitzian curve $x(t)$ such that, for any $0<t<T$, we have both
\begin{eqnarray*}
\int_0^\infty |U(t,x+x(t))- U_R|^2\,dx\leq C\eps(1+t),\qquad
\int_{-\infty}^0 |U(t,x+x(t))- U_L|^2\,dx\leq \eps^4,
\end{eqnarray*}
(resp. $\int_0^\infty |U(t,x+x(t))- U_R|^2\,dx\leq \eps^4 $ and $\int_{-\infty}^0 |U(t,x+x(t))- U_L|^2\,dx\leq C\eps(1+t)$).
In both cases we have 
$$
|x(t)-\sigma t|\leq C\sqrt{\eps t(1+t)}.
$$
\end{theo}

In particular, this provides $L^2$ stability, up to a drift, for suitably weak shocks and contact discontinuities in a class of perturbations without $BV$ conditions. Note that the assumption of genuinely nonlinearity or linear degeneracy applies only to the wave family associated to the extremal eigenvalue. No such assumptions are needed on the other wave families.
\vskip0.3cm

The theorems above highlight only a few applications of our theory. 
In the next section, we develop our methods in a more general framework.
The assumptions on the Hugoniot curves are quite natural and we require no smallness condition on the discontinuities at play. We can even relax the strict hyperbolicity condition and consider cases in which the middle eigenvalues degenerate and possibly cross each other.

\vskip0.3cm To conclude this introduction, let us mention that the relative entropy method is also an important tool in the study of asymptotic limits to conservation laws. Applications of the relative entropy method in this context began with the work of Yau \cite{Yau} and have been studied by many others. For incompressible limits, see Bardos, Golse, Levermore \cite{Bardos_Levermore_Golse1,Bardos_Levermore_Golse2}, Lions and  Masmoudi \cite{Lions_Masmoudi}, Saint Raymond et al. \cite{SaintRaymond1,SaintRaymond2,SaintRaymond3,SaintRaymond4}. For the compressible limit, see Tzavaras \cite{Tzavaras_theory} in the context of relaxation 
and \cite{BV,BTV,MV} in the context of hydrodynamical limits. Up to now, this method works as long as the limit solution is Lipschitz. It would be of significant
interest to extend the method to shocks (see \cite{Vasseur_Book}).

\section{Presentation of the results}

\subsection{General framework}

We want to study a system of $m$ equations of the form
\begin{equation}\label{system}
\dt U + \dx A(U)=0,
\end{equation}
where the flux function $A$ is defined on an open, bounded, convex set $\mV \subset \R^m$.
$$
A: \mV\subset \R^m\longrightarrow \R^m.
$$
We assume that $A \in C^2(\mV)$. Additionally, we assume the existence of a strictly convex entropy
$$
\eta: \mV\subset \R^m\longrightarrow \R,
$$
of class $C^2$, and an associated entropy flux
$$
G: \mV\subset \R^m\longrightarrow \R,
$$
of class $C^2$, such that the following compatibility relation holds on $\mV$.
\begin{equation}\label{entropy flux}
\partial_j G=\sum_{i=1}^m \partial_i \eta \thinspace \partial_j A_i \qquad \mathrm{for \  any \ } 1\leq j\leq m.
\end{equation}

If we want to apply our theory to the systems of gas dynamics, we have to define these functions on a suitable subset of the boundary of $\mV$, namely the points corresponding to vacuum states. 
For this reason, we introduce 
$$
\mU=\overline{\mV},
$$
the closure of $\mV$. We assume, for simplicity, that $A$, $\eta$, and $G$ are continuous on $\mU$ (but with no additional regularity up to the boundary). We denote by $\mU^0$ the subset of $\mU$ where at least one of the functions $\eta$, $A$, $G$ is not $C^1$ (typically the vacuum states). 
\vskip0.5cm
\noindent {\bf Remark.} It is possible to consider a more general situation in which $\eta$ is unbounded on $\mV$. In that case, we can add the "vacuum points" in the following way, as in \cite{Vasseur_Book}.
\begin{equation*}
\mU=\{V\in\R^m \ \vert \ \ \exists V_k\in \mV, \ \lim_{k\to\infty} V_k=V, \ \limsup_{k\to\infty} \ \eta(V_k)<\infty\}.
\end{equation*}
We consider here a slightly less general framework, which is enough to treat the Euler systems on a state domain with bounded velocities.
\vskip0.5cm
Next, we define, for any $V\in \mV$, $U\in\mU$, the relative entropy function
$$
\eta(U \midd V)=\eta(U)-\eta(V)- \nabla \eta(V)\cdot(U-V).
$$
Since $\eta$ is continuous (in fact, convex) on $\mU$ and strictly convex in $\mV$, we have
$$
\eta(U \midd V)\geq 0,\qquad U\in\mU, \ V\in \mV,
$$
and
$$
\eta(U \midd V)=0 \quad \mathrm{if \ and \ only \ if} \quad U=V.
$$
Indeed, the following lemma shows that the relative entropy is comparable to the square of the $L^2$ norm.
\begin{lemm}\label{L2}
For any compact set $\Omega \subset \mV$, there exist $C_1, C_2 > 0$ such that 
$$
C_1|U-V|^2\leq\eta(U \midd V)\leq C_2|U-V|^2,
$$
for any $U\in \mU$ and  $V \in \Omega$.
\end{lemm}
We give a proof of this well-known estimate in the appendix.
\vskip0.3cm

For a pair of states $U_L \ne U_R$ in $\mV$, we say that $(U_L,U_R)$ is an entropic Rankine-Hugoniot discontinuity if there exists $\sigma\in \R$ such that 
\begin{equation}\label{RH}
\begin{array}{l}
\ds{A(U_R)-A(U_L)=\sigma(U_R-U_L),}\\[0.3cm]
\ds{G(U_R)-G(U_L)\leq\sigma(\eta(U_R)-\eta(U_L)).}
\end{array}
\end{equation}
Equivalently, this means that the discontinuous function $U$ defined by
\begin{align*}
U(t,x) = 
\begin{cases}
U_L, &\text{if $x<\sigma t$,}\\[0.1 cm]
U_R, &\text{if $x>\sigma t$,}\\
\end{cases}
\end{align*}
is a weak solution to (\ref{system}) verifying also, in the sense of distributions, the entropy inequality
\begin{equation}\label{entropie}
\dt \eta(U)+\dx G(U)\leq 0.
\end{equation}

Finally, we denote by $\lambda^-(U)$ and $\lambda^+(U)$ the smallest and largest eigenvalues, respectively, of $\nabla A(U)$. Hereafter, we assume that  $\lambda^\pm(U)$ are simple eigenvalues for all $U\in \mV$ and that $U\to\lambda^\pm(U)$ lie in $L^\infty(\mU)$. (They may be undefined on $\mU^0$.)

\subsection{Hypotheses on the system}\label{hypos}

First, we assume that for any $(U_-, U_+)$ entropic Rankine-Hugoniot discontinuity with $U_-\neq U_+$ we have both $U_-\notin \mU^0$ and   $U_+\notin \mU^0$. (Typically, there is no shock connecting the vacuum.) 
\vskip0.3cm
We will consider two sets of assumptions. One set will imply the result on the 1-shock (or 1-contact discontinuity), the second set (dual from the first one)
will imply the result on the n-shock (or n-contact discontinuity). A system satisfying both set of hypotheses, verifies both results. 
\vskip0.4cm
\noindent{\bf First set of hypotheses}
\vskip0.1cm 

The first set of hypotheses, related to some  $U_L\in \mV$, is the following ((H1) to (H3)). 
\begin{itemize}
\item[(H1)] (Family of 1-contact discontinuities or 1-shocks verifying the Liu condition)\\
There exists a neighborhood $B\subset\mV$ of $U_L$ such that for any $U\in B$, there is a one parameter family of states $S_U(s)\in \mU$ defined on an interval $[0,s_U]$,  such that  $S_{U}(0)=U$, and 
$$
A(S_U(s))-A(U)=\sigma_U(s)(S_U(s)-U), \qquad s\in [0,s_U],
$$
(which means that $(U,S_U(s))$ is a Rankine-Hugoniot discontinuity with velocity $\sigma_U(s)$). We assume that $U\to s_U$ is Lipschitz on $B$ and 
both $(s,U)\to S_U(s)$ and $(s,U)\to\sigma_U(s)$ are $C^1$ on $\{(s,U) | U\in B, 0\leq s\leq s_U\}$. We assume also that the following properties hold for $U\in B$.
\begin{itemize}
\item[(a)] $\sigma_U'(s)\leq 0$ for $0\leq s \leq s_U$ (the speed of the shock decreases with $s$), and $\sigma_U(0) = \lambda^-(U)$.
\item[(b)] (1-shock) If $\sigma_U' \not\equiv 0$, then  $\ds{\frac{d}{ds}\eta(U | S_U(s))}\geq0$ (the shock "strengthens" with $s$) for all $s$.
\end{itemize}
\item[(H2)]  
If $(U,V)$ is an entropic Rankine-Hugoniot discontinuity with velocity $\sigma$ such that $V \in B$, then $\sigma\geq\lambda^-(V)$.
\item[(H3)]
If $(U,V)$ is an entropic Rankine-Hugoniot discontinuity with velocity $\sigma$ such that $U \in B$ and $\sigma < \lambda^-(U)$, then $(U,V)$ is a 1-shock. In particular, $V = S_U(s)$ for some $0\leq s\leq s_U$.
\end{itemize}
\vskip0.4cm

\noindent{\bf Second set of hypotheses}
\vskip0.1cm

The second set of hypotheses, related to some  $U_R\in \mV$, is the following (($\text{H}^{\prime}$1)  to ($\text{H}^{\prime}$3)). 
\begin{itemize}
\item[($\text{H}^{\prime}$1)] (Family of $n$-contact discontinuities or $n$-shocks verifying the Liu condition)\\
There exists a neighborhood $B\subset\mV$ of $U_R$  such that  for every $U\in B$ there is a one parameter family of states $S_U(s)\in \mU$ defined on an interval $[0,s_U]$,  such that $S_{U}(0)=U$, and 
$$
A(S_U(s))-A(U)=\sigma_U(s)(S_U(s)-U), \qquad s\in [0,s_U],
$$
(which means that $(S_U(s),U)$ is a Rankine-Hugoniot discontinuity with velocity $\sigma_U(s)$). We assume that $U\to s_U$ is Lipschitz and 
both $(s,U)\to S_U(s)$ and $(s,U)\to\sigma_U(s)$ are  $C^1$ on $\{(s,U)| U\in B, s\in [0,s_U]\}$. We assume also that the following properties hold for $U\in B$.
\begin{itemize}
\item[(a)] $\sigma_U'(s)\geq 0$ for $0\leq s\leq s_U$ (the speed of the shock increases with $s$), and $\sigma_U(0)=\lambda^+(U)$.
\item[(b)] ($n$-shock) If $\sigma_U' \not\equiv 0$, then  $\ds{\frac{d}{ds}\eta(U | S_U(s))}\geq0$ (the shock ''strengthens''  with $s$) for all $s$.
\end{itemize}
\item[($\text{H}^{\prime}$2)] 
If $(V,U)$ is an entropic Rankine-Hugoniot discontinuity with velocity $\sigma$ such that $V \in B$, then $\sigma\leq\lambda^+(V)$.
\item[($\text{H}^{\prime}$3)] 
If $(V,U)$ is an entropic Rankine-Hugoniot discontinuity with velocity $\sigma$ such that $U \in B$ and $\sigma > \lambda^+(U)$, then $(V,U)$ is an n-shock. In particular, $V = S_U(s)$ for some $0\leq s\leq s_U$.
\end{itemize}
\vskip0.5cm
\noindent{\bf Remarks}
\vskip0.1cm

\begin{itemize}
\item Note that a given system (\ref{system}) verifies Properties (H1) to (H3) relative to $U \in \mV$ if and only if  the system 
\begin{equation}\label{reverse system}
 \dt U-\dx A(U)=0,
\end{equation}
verifies Properties ($\text{H}^{\prime}$1) to ($\text{H}^{\prime}$3) relative to the same $U \in \mV$. The properties are, in this way, dual. 
\item In the case $\sigma^{\prime}_U(s)=0$ for all $s$, Property (H1) just assumes the existence of a 1-family of contact discontinuities.
\item In the case where $\sigma_U(s)$ is not constant in $s$, Property (H1) assumes the existence of a family a 1-shocks $(U, S_U(s))$ verifying the Liu entropy condition for all $s$ (Property (a)). The only additional requirement is (b), which is a condition on the growth of the shock along $S_U(s)$, where the growth is measured
through the pseudo-metric induced from the entropy. 
This condition arises naturally in the study of admissibility criteria for systems of conservation laws. In particular, it ensures that Liu admissible shocks are entropic even for moderate to strong shocks.
Indeed, this fact follows from the important formula
\begin{align*}
G(S_{U_L}(s))-G(U_L) = \sigma_{U_L}(s) \thinspace (\eta(S_{U_L}(s)) - \eta(U_L)) + \int_0^s\sigma'_{U_L}(\tau)\eta(U_L \midd S_{U_L}(\tau))\,d\tau,
\end{align*}
which is proved in Section \ref{structure}. (See also \cite{Dafermos4,Lax,Ruggeri}.)

\item Hypothesis (H2) is fulfilled under the very general assumption that all the entropic Rankine-Hugoniot discontinuities verify the Lax entropy conditions, that is
$$
\lambda_i(U_-) \geq \sigma \geq \lambda_i(U_+),
$$
for any $i$-shocks $(U_-,U_+)$ with velocity $\sigma$ and any $1\leq i\leq n$. Indeed, we need only the second inequality, and the fact that 
$\lambda_i(U_+)\geq \lambda^-(U_+)$.

\item Hypothesis (H3) is a requirement that the family of 1-discontinuities is well-separated from the other Rankine-Hugoniot discontinuities and do not interfere with them. In the case of strictly hyperbolic systems,
it is, for instance, a consequence of the extended Lax admissibility condition
$$
\lambda_{i+1}(U_+) \geq \sigma \geq \lambda_{i-1}(U_-),
$$
for all i-shocks $(U_-,U_+)$, $i>1$, with velocity $\sigma$.
Indeed, we use only the second inequality and the fact that $\lambda_{i-1}(U_-)\geq\lambda^-(U_-)$. Note that we need to separate only the $1$-shocks
issued from $B$, that is close to $U_L$. 
\item The existence of an entropy $\eta$ implies that the system (\ref{system}) is hyperbolic. Since $A\in C^2(\mV)$, the eigenvalues of $\nabla A(U)$ vary continuously on $\mV$. In particular, since $\lambda^\pm(U)$ are simple for $U\in\mV$, the implicit function theorem ensures that the maps $U\to\lambda^\pm(U)$ are in $C^1(\mV)$. 
Note, however, that those maps may be discontinuous on $\mU$.
\end{itemize}

\subsection{Statement of the result}

Our main result is the following.
\begin{theo}\label{main}
Consider a system of conservation laws (\ref{system}), such that $A$ is $C^2$ on  an open, bounded, convex subset  $\mV$ of $\R^m$.
We assume that there exists  a $C^2$ strictly convex entropy $\eta$ on $\mV$ verifying (\ref{entropy flux}). We assume that $\eta$, $A$ and $G$ are continuous on $\mU=\overline{\mV}$. Let $U_L\in \mV$. Assume that the system (\ref{system}) verifies the Properties (H1)--(H3). Consider $U_R\in\mV$ such that 
  $(U_L,U_R)$ is a 1-shock (or 1-contact discontinuity) with velocity $\sigma$. (This means that there exists $s>0$ such that $U_R=S_{U_L}(s)$ and $\sigma=\sigma_{U_L}(s)$).   Then there exist constants $C>0$, $\eps_0>0$ such that for any $0<\eps<\eps_0$ and any weak entropic solution $U$ of (\ref{system}) with values in $\mU$ on $(0,T)$ verifying the strong trace property of Definition \ref{defi_trace}, and 
\begin{eqnarray*}
&&\int_0^\infty |U_0(x)- U_R|^2\,dx\leq \eps,\qquad
\int_{-\infty}^0|U_0(x)- U_L|^2\,dx\leq \eps^4,
\end{eqnarray*}
there exists a Lipschitzian map $x(t)$ such that for any $0<t<T$:
\begin{eqnarray*}
&&\int_0^\infty |U(t,x+x(t))-U_R|^2\,dx\leq C\eps(1+t),\qquad
\int_{-\infty}^0 |U(t,x+x(t))- U_L|^2\,dx\leq \eps^4.
\end{eqnarray*}
Moreover,
$$
|x(t)-\sigma t|\leq C\sqrt{\eps t(1+t)}.
$$
\end{theo}
This shows that the theory is essentially an $L^2$ theory. The correction of the position of the approximated shock $x(t)$ is fundamental, since the 
result is trivially wrong without it, even for Burgers' equation in the scalar case (see \cite{Leger}). Part of the difficulty of the proof is to find this correct position. Indeed, we will construct it so that $\int_{-\infty}^0\eta(U(t,x+x(t))|U_L)\,dx$ is decreasing and remains smaller than $\eps^4$ for all time.
\vskip0.3cm
Applying the main theorem on the $\tilde{U}(t,x)=U(t,-x)$, which is an entropic solution to (\ref{reverse system}), we obtain the following corollary.
\begin{corr}\label{corr_L2_reverse}
Consider a system of conservation laws (\ref{system}), such that $A$ is $C^2$ on  a open, bounded, convex subset  $\mV$ of $\R^m$.
We assume that there exists  a $C^2$ strictly convex entropy $\eta$ on $\mV$ verifying (\ref{entropy flux}). We assume that $\eta$, $A$ and $G$ are continuous on $\mU=\overline{\mV}$. Let $U_R\in \mV$. Assume that (\ref{system}) verifies the properties ($H^{\prime}$1)--($H^{\prime}$3). Let $U_L\in\mV$ such that 
  $(U_L,U_R)$ is a $n$-shock (or $n$-contact discontinuity) with velocity $\sigma$.   (This means that there exists $s>0$ such that $U_L=S_{U_R}(s)$ and $\sigma=\sigma_{U_R}(s)$).  Then there exist constants $C>0$, $\eps_0>0$ such that for any $0<\eps<\eps_0$ and any weak entropic solution $U$ of (\ref{system}) with values in $\mU$ on $(0,T)$ verifying the strong trace property of Definition \ref{defi_trace}, and 
\begin{eqnarray*}
&&\int_0^\infty |U_0(x)- U_R|^2\,dx\leq \eps^4,\qquad
\int_{-\infty}^0|U_0(x)- U_L|^2\,dx\leq \eps,
\end{eqnarray*}
there exists a Lipschitzian map $x(t)$ such that for any $0<t<T$:
\begin{eqnarray*}
&&\int_0^\infty |U(t,x+x(t))-U_R|^2\,dx\leq \eps^4,\qquad
\int_{-\infty}^0 |U(t,x+x(t))- U_L|^2\,dx\leq C\eps(1+t).
\end{eqnarray*}
Moreover,
$$
|x(t)-\sigma t|\leq C\sqrt{\eps t(1+t)}.
$$
\end{corr}
In particular, if we consider a $2\times2$ system which verifies both (H1)--(H3) for any $U_L\in\mV$ and ($\text{H}^{\prime}$1)--($\text{H}^{\prime}$3) for any $U_R\in\mV$, then
all shocks (and contact discontinuities) are unique and stable in $L^2$. 
\vskip0.3cm
Note that the assumptions on the system are quite minimal. There is an assumption only on the wave coming from $U_L$ (or going to $U_R$ in the case
of the corollary). There are absolutely no assumptions on the other waves (which may not even exist or may be neither genuinely nonlinear nor linearly degenerate). The extremal property that the shock curve under consideration corresponds to the smallest eigenvalue of $\nabla A(U)$ (resp. the largest 
eigenvalue of $\nabla A(U)$) is only prescribed on a small neighborhood of $U_L$ (resp. of $U_R$). 
Finally the theory allows us (via the extended set $\mU$) to consider weak solutions which
may take values $U$ corresponding to points of non-differentiability of $A$ and $\eta$. This includes, for example, the vacuum states in fluid mechanics.

\subsection{Main ideas of the proof}

The main idea of the proof is to find a path $t \to x(t)$ for which the total relative entropy of $U(t,x)$ with respect to $U_L$ on the left of $x(t)$, given by 
$$
\int_{-\infty}^0\eta(U(t,x+x(t)) \midd U_L)\,dx,
$$
is strictly decreasing. The following estimate underlies most of our analysis.
\begin{lemm}\label{defi_F}
If $V\in\mV$ and $U$ is any weak entropic solution of (\ref{system}), then $\eta(U \midd V)$ is a solution in the sense of distributions to
$$
\dt \eta(U \midd V)+\dx F(U, V)\leq 0,
$$
where 
$$
F(U,V)=G(U)-G(V)-\nabla \eta(V) \cdot (A(U)-A(V)).
$$
\end{lemm}
The proof of this lemma is direct from the definition of the relative entropy (Note that $V$ is constant with respect to $t$ and $x$). From this lemma, and the strong trace property of Definition \ref{defi_trace}, we will show that
$$
\frac{d}{dt} \int_{-\infty}^0\eta(U(t,x+x(t)) \midd U_L)\,dx \leq x'(t)\eta(U(t,x(t)-) \midd U_L)-F(U(t,x(t)-),U_L),
$$
for almost every $t$. Let us fix a small drift velocity $v_0 >0$ of order $\eps$. Ideally, we would like to have 
\begin{equation}\label{ideal}
x'(t)=\frac{F(U(t,x(t)-),U_L)}{\eta(U(t,x(t)-) \midd U_L)}-v_0.
\end{equation}
This would preserve the estimate
\begin{align*}
\int_{-\infty}^0 \eta(U(t,x+x(t)) \midd U_L)\,dx\leq \eps^4,
\end{align*}
which is assumed at the initial time, and ensure that
\begin{align*}
\frac{d}{dt}\int_{-\infty}^0\eta(U(t,x+x(t)) \midd U_L)\,dx\leq -v_0\eps^2, \qquad \text{whenever $\eta(U(t,x(t)-) \midd U_L)\geq \eps^2$}.
\end{align*}
If we denote by $\mathcal{I}$ the set of time for which $\eta(U(t,x(t)-) \midd U_L)\geq \eps^2$, then we must have
$$
|\mathcal{I}|\leq \frac{\eps^4}{v_0 \eps^2}\sim\eps.
$$

Before going on, let us note that $U$ has little regularity, and (\ref{ideal}) cannot be solved in the classical sense. Hence we can define $x(t)$ only in the Filippov way. (This is roughly the idea. For technical reasons, the construction is even more complicated than that.) We will have to check carefully that we can do it using only the strong trace property of Definition \ref{defi_trace}. Even so, we cannot ensure that (\ref{ideal}) holds almost everywhere. However, we will use the fact that for almost every time $t$, especially when $U(t,x(t)+) \neq U(t,x(t)-)$, the following Rankine--Hugoniot relation holds:
$$
A(U(t,x(t)+)-A(U(t,x(t)-)=x'(t)(U(t,x(t)+)-U(t,x(t)-)).
$$ 
We will make these claims more precise in the sections to follow.

At this point we need to check that the total relative entropy of $U(t,x)$ with respect to $U_R$ on the right of $x(t)$ remains under control. A similar estimate to the one above gives
\begin{align*}
\int_{0}^{\infty}\eta(U(t,x+x(t)) \midd U_R)\,dx &=\int_{0}^{\infty}\eta(U_0(x) \midd U_R)\,dx\\[0.1 cm]
& \qquad +\int_0^t \left\{ F(U(t,x(t)+),U_R)-x'(t)\eta(U(t,x(t)+) \midd U_R) \right\}\,dt.
\end{align*}
Since $U$ is bounded, we have the easy estimate
$$
\int_{\mathcal{I}} \left\{ F(U(t,x(t)+),U_R)-x'(t)\eta(U(t,x(t)+) \midd U_R) \right\} \,dt\leq C\eps.
$$
In the remaining case, we have to control $F(U(t,x(t)+),U_R)-x'(t)\eta(U(t,x(t)+) \midd U_R)$ for $t$ such that $\eta(U(t,x(t)-) \midd U_L)\leq \eps^2$, that is 
$$
|U(t,x(t)-)-U_L|\sim \eps.
$$
The extremality property of $(U_L,U_R)$ ensures that the worst scenario  corresponds to
$$
U(t,x(t)+)=S_{U(t,x(t)-)}(s), \qquad x'(t)=\sigma_{U(t,x(t)-)}(s) \qquad \mathrm{for \  given \ } s>0.
$$
But, $S_{U(t,x(t)-)}(s) =S_{U_L}(s) + \mathcal{O}(\eps)$, and the key structural lemma will show that 
$$
F(S_{U_L}(s),U_R)-\sigma_{U_L}(s)\eta(S_{U_L}(s) \midd U_R) \leq 0,\qquad \mathrm{for} \ s>0.
$$
Hence, for $t\notin \mathcal{I}$, 
$$
F(U(t,x(t)+),U_R)-x'(t)\eta(U(t,x(t)+) \midd U_R)\sim \eps.
$$
\vskip0.3cm
The rest of the paper is organized as follows. In the next section, we prove the main structural lemmas. 
In the following section we construct the path $t\to x(t)$. The next one is dedicated to the proof of the main theorem. In the last section, we prove
theorems on the example stated in the introduction. 

\section{Structural lemmas}\label{structure}

The first lemma of this section gives an explicit formula for the entropy lost at a Rankine--Hugoniot discontinuity $(U_-,U_+)$, where $U_+=S_{U_-}(s)$ for some $s>0$. The estimate can be traced back to the work of Lax \cite{Lax}.

\begin{lemm}\label{decreasing}
Assume $(U_-,U_+)\in \mV^2$ is an entropic Rankine-Hugoniot discontinuity with velocity $\sigma$; that is, $(U_-,U_+)$ verifies (\ref{RH}). Then, for any $V\in \mU$
$$
F(U_+,V)-\sigma \eta(U_+ \midd V)\leq F(U_-,V)-\sigma\eta(U_- \midd V),
$$
where $F$ is defined as in Lemma \ref{defi_F}. Furthermore, if $U_-\in B$, as in Hypothesis (H1), and there exists $s>0$ such that $U_+=S_{U_-}(s)$ and $\sigma=\sigma_{U_-}(s)$ (that is, $(U_-,U_+)$ is a $1$-discontinuity), then
$$
F(U_+,V)-\sigma \eta(U_+ \midd V)= F(U_-,V)-\sigma\eta(U_- \midd V)+\int_0^s\sigma'_{U_-}(\tau)\eta(U_- \midd S_{U_-}(\tau))\,d\tau.
$$
\end{lemm}
\begin{proof}
Since $(U_-,U_+)\in \mV^2$ is an entropic Rankine-Hugoniot discontinuity with velocity $\sigma$ we have
$$
-\nabla \eta(V) \cdot (A(U_+)-A(U_-))=-\sigma \nabla \eta(V) \cdot (U_+-U_-),
$$
and 
$$
G(U_+)-G(U_-)\leq \sigma (\eta(U_+)-\eta(U_-)).
$$
Summing those two estimates gives the first result.

Assume now that it is a $1$-discontinuity.
Then, define
\begin{align*}
\mathcal{F}_1(s)&=F(S_{U_-}(s),V)-F(U_-,V),\\
\mathcal{F}_2(s)&=\sigma(s)(\eta(S_{U_-}(s) \midd V)-\eta(U_- \midd V))+\int_0^s\sigma'_{U_-}(\tau)\eta(U_- \midd S_{U_-}(\tau))\,d\tau.
\end{align*}
We want to show that $ \mathcal{F}_1(s)= \mathcal{F}_2(s)$ for all $s$.
Since $S_{U_-}(0)=U_-$, the equality is true  for $s=0$.
Next we have
\begin{align*}
\mathcal{F}'_1(s)&=\frac{d}{ds}G(S_{U_-}(s))-\nabla \eta(V) \cdot \frac{d}{ds}A(S_{U_-}(s))\\[0.1 cm]
&= [ \nabla \eta(S_{U_-}(s))-\nabla \eta(V) ] \cdot \frac{d}{ds}[A(S_{U_-}(s))-A(U_L)],
\end{align*}
and 
\begin{align*}
\mathcal{F}'_2(s)&=\sigma_{U_-}'(s)[\nabla \eta(V) \cdot ((S_{U_-}(s)-V)-(U_--V))-\nabla \eta(S_{U_-}(s)) \cdot (S_{U_-}-U_-)]\\[0.1 cm]
&\qquad \qquad \qquad \qquad \qquad +\sigma_{U_-}(s)[\nabla \eta(S_{U_-}(s))-\nabla \eta(V)] \cdot S_{U_-}'(s)\\[0.2 cm]
&=[\nabla \eta(S_{U_-}(s))-\nabla \eta(V)] \cdot \frac{d}{ds}[\sigma_{U_-}(s) (S_{U_-}(s)-U_L)].
\end{align*}
Using the fact that $(U_-,S_{U_-}(s))$ with velocity $\sigma_{U_-}(s)$ verifies the Rankine-Hugoniot conditions, we get
$$
\mathcal{F}'_1(s)=\mathcal{F}'_2(s)\qquad \mathrm{for\ } s>0.
$$
\end{proof}
The next lemma is fundamental to the theory. It is a slight variation on a crucial
lemma of DiPerna \cite{DiPerna}.
\begin{lemm}\label{cornerstone}
For any $U\in B$ and any $s, s_0 \in [0, s_U]$, we have
$$
F(S_U(s),S_U(s_0))-\sigma_U(s)\eta(S_U(s) \midd S_U(s_0))=\int_{s_0}^{s}\sigma_U'(\tau)(\eta(U|S_U(\tau))-\eta(U \midd S_{U}(s_0)))\,d\tau\leq0.
$$
\end{lemm}
\begin{proof}
We use the estimate of Lemma \ref{decreasing} twice with $V=S_U(s_0)$ and $U_-=U$. The first time we take  $U_+=S_U(s)$, and the second time  $U_+=S_{U}(s_0)$.
The difference of the two results gives the lemma.
\end{proof}

\section{Construction of $t\to x(t)$}
In this section, we contruct the path $t\to x(t)$ and study its properties. Unlike the construction in \cite{Leger}, we do not appeal directly to the theory of Filippov. Instead, we build $x(t)$ as the limit of approximate solutions, and then show that the Filippov properties are satisfied. The procedure is interesting in its own right and has similarities with the method of Hermes (see \cite{Hajek}). We will focus on the family of 1-discontinuities. The corresponding results are easily obtained for the family of n-discontinuities. 

Throughout this section, we assume that $(U_L,U_R)\in\mV^2$ is a fixed 1-discontinuity with velocity $\sigma$, and that $U$ is a fixed weak entropic solution of (\ref{system}) verifying the strong trace property of Definition \ref{defi_trace}. We also fix $\eps>0$. As before, $\lambda^-(U)$ denotes the smallest eigenvalue of $\nabla A(U)$ and is defined for all $U\in \mU\setminus \mU^0$. We define on $\mU$ the velocity function
\begin{align}\label{velocity}
V(U)=
\begin{cases}
\ds{\min\left\{\frac{F(U,U_L)}{\eta(U \midd U_L)}-\eps, \thinspace \lambda^-(U)-\eps\right\}}, &\text{if \ $U\in \mU\setminus(\mU^0\cup\{U_L\})$,}\\[0.5 cm]
\lambda^-(U_L)-\eps, &\text{if \ $U=U_L$,}\\[0.3 cm]
\ds{\frac{F(U,U_L)}{\eta(U \midd U_L)}-\eps}, &\text{if \ $U\in\mU^0$.}\\
\end{cases}
\end{align}
We have the following lemma.
\begin{lemm}\label{lemmdisc}
The function $U\to V(U)$   is continuous on $\mU\setminus\mU^0$, and is bounded and upper semi-continuous on $\mU$. Moreover, there exists a constant $C$ (independent of $\eps$)
such that if $U\in B$ and $\eta(U \midd U_L)\leq \eps^2$, then
$$
V(U)\geq \lambda^-(U_L)-C\eps.
$$
\end{lemm}
Note that this map may not be continuous  on $\mU^0$.
\begin{proof}
First note that the function $U\to \frac{F(U,U_L)}{\eta(U \midd U_L)}$ is continuous on $\mU\setminus \{U_L\}$. Also, by assumption, $\lambda^-(U)$ is continuous on $\mU\setminus \mU^0$ and bounded on $\mU$. So $U\to V(U)$ is continuous on $\mU\setminus(\{U_L\}\cup \mU^0)$, and upper semi-continuous  on $\mU\setminus\{U_L\}$. 
At the point $U_L$, 
we have
\begin{eqnarray*}
&&\nabla_UF(U,U_L) \big \vert_{U = U_L}=0,\qquad \nabla^2_UF(U,U_L)\big \vert_{U = U_L} =D^2 \eta(U_L) \nabla A(U_L),\\
&&\nabla_U\eta(U \thinspace | \thinspace U_L)\big \vert_{U = U_L}=0,\qquad \nabla^2_U\eta(U \thinspace | \thinspace U_L)\big \vert_{U = U_L} =D^2 \eta (U_L).
\end{eqnarray*}
Owing to the strict convexity of $\eta$ at $U_L$, an expansion near $U_L$ gives
\begin{equation}\label{O}
\frac{F(U,U_L)}{\eta(U \midd U_L)}=\frac{(U-U_L)^T D^2 \eta(U_L) \nabla A(U_L)(U-U_L)}{(U-U_L)^T D^2 \eta(U_L)(U-U_L)}+\mathcal{O}(|U-U_L|).
\end{equation}
Since $D^2 \eta(U_L)$ is symmetric positive definite and $D^2 \eta(U_L) \nabla A(U_L)$ is symmetric, those two matrices are diagonalizable in the same basis.
This gives
\begin{align}\label{O2}
\lambda^-(U_L) D^2 \eta(U_L) \leq D^2 \eta(U_L) \nabla A(U_L) \leq \lambda^+(U_L) D^2 \eta(U_L).
\end{align}
Therefore,
$$
\liminf_{U\to U_L} \thinspace \frac{F(U,U_L)}{\eta(U \midd U_L)}-\eps\geq \lambda^-(U_L)-\eps=V(U_L),
$$
which implies
$$
\lim_{U\to U_L}V(U)= \lambda^-(U_L)-\varepsilon.
$$
So, $V$ is continuous at $U=U_L$ and bounded on $\mU$.
\vskip0.3cm
More precisely, (\ref{O}) and (\ref{O2}) give that 
$$
\frac{F(U, U_L)}{\eta(U \midd U_L)}\geq \lambda^-(U_L)-\mathcal{O}(|U-U_L|).
$$
Since $U\to\lambda^-(U)$ lies in  $C^1(B)$,
we have also that
$$
\lambda^-(U)\geq \lambda^-(U_L)-\mathcal{O}(|U-U_L|).
$$
This provides the last statement of the lemma.
\end{proof}
For any Lipschitzian path  $t\to x(t)$ we define
\begin{eqnarray*}
\Vs(t)&=&\max\left\{V(U(t,x(t)-)), \thinspace V(U(t,x(t)+))\right\},\\[0.2 cm]
\Vi(t)&=&
\begin{cases}
\min\left\{V(U(t,x(t)-)), \thinspace V(U(t,x(t)+))\right\}, &\text{if both $U(t,x(t)-), \ U(t,x(t)+) \notin\mU^0$},\\
-\infty, &\text{otherwise.}\\
\end{cases}
\end{eqnarray*}

This section is dedicated to the following proposition.
\begin{prop}\label{x(t)}
For any $(U_L,U_R)\in\mV^2$    1-discontinuity with velocity $\sigma$,  $\eps>0$,
and $U$ a weak entropic solution of (\ref{system}) verifying the strong trace property of Definition \ref{defi_trace}, there exists a Lipschitzian path $t\to x(t)$
such that  for almost every $t>0$
$$
\Vi(t)\leq x'(t)\leq \Vs(t).
$$ 
\end{prop}
\begin{proof}
Consider the function
$$
v_n(t,x)=\int_{0}^1V(U(t,x+\frac{y}{n}))\,dy.
$$
By virtue of Lemma \ref{lemmdisc}, $v_n$ is bounded, measurable in $t$, and Lipschitz in $x$. We denote by $x_n$ the unique solution to 
\begin{equation*}
\left\{\begin{array}{l}
x_n'(t)=v_n(t,x_n(t)), \qquad t>0,\\[0.1cm]
x_n(0)=0,
\end{array}\right.
\end{equation*}
in the sense of Carath\'{e}odory.
Since $v_n$ is uniformly bounded, $x_n$ is uniformly Lipschitzian (in time) with respect to $n$. Hence, there exists a Lipschitzian path $t\to x(t)$ such that (up to a subsequence) $x_n$ converges to  $x$ in $C^0(0,T)$ for every $T>0$.
We construct $\Vs(t)$ and $\Vi(t)$ as above for this particular fixed path $t\to x(t)$. Let us show that for almost every $t>0$
\begin{align}
&\ds{\lim_{n\to\infty} \ [x'_n(t)-\Vs(t)]_+ =0,} \label{est1}\\[0.1cm]
&\ds{\lim_{n\to\infty} \ [\Vi(t) - x'_n(t) ]_+=0.} \label{est2}
\end{align}
Both limits can be proved the same way. Let us focus on the first one. We have
\begin{align*}
[x'_n(t)-\Vs(t)]_+ &= \left[\int_0^1V(U(t,x_n(t)+\frac{y}{n}))\,dy-\Vs(t)\right]_+\\[0.1 cm]
&=\left[\int_0^1[V(U(t,x_n(t)+\frac{y}{n}))-\Vs(t)]\,dy\right]_+\\[0.1 cm]
&\leq \int_0^1\left[V(U(t,x_n(t)+\frac{y}{n}))-\Vs(t)\right]_+\,dy\\[0.1 cm]
&\leq \ \esssup_{y \in (0,1)} \ \left[V(U(t,x_n(t)+\frac{y}{n}))-\Vs(t)\right]_+ \\[0.1 cm]
&\leq \esssup_{z \in (-\eps_n, \eps_n)} \ \left[V(U(t,x(t)+z))-\Vs(t)\right]_+,
\end{align*}
where, for a given $t>0$, $\eps_n \to 0$ is chosen so that $x_n(t)- x(t) \in (-\eps_n, \eps_n - \frac{1}{n})$. 
We claim that for almost every $t>0$, the last term above goes to zero as $n \to \infty$. Indeed, fix $t >0$ for which $U(t, x(t) + \cdot)$ has a left and right trace in the sense of Definition \ref{defi_trace}. That is,
\begin{align*}
\lim_{\eps \to 0} \left\{ \esssup_{y \in (0,\eps)} \ds \vert U(t, x(t) + y) - U_+(t) \vert  \right\} = \lim_{\eps \to 0} \left\{ \esssup_{y \in (0, \eps)} \ds \vert U(t, x(t) - y) - U_-(t) \vert  \right\}= 0,
\end{align*}
Since $V$ is upper semi-continuous on $\mU$, we have that for all $r >0$ there exists $\delta > 0$ such that 
\begin{align*}
\vert U - U_\pm(t) \vert < \delta \quad \Rightarrow  \quad  [ V(U) - V(U_\pm(t)) ]_+ < r.
\end{align*}
Therefore,
\begin{align*}
\lim_{\eps \to 0} \left\{ \esssup_{y \in (0,\eps)} \ds \left[ V(U(t, x(t) \pm y)) - V(U_\pm(t)) \right]_+  \right\} = 0,
\end{align*}
and it follows easily that 
\begin{align*}
\lim_{\eps \to 0} \left\{ \esssup_{z \in (-\eps,\eps)} \ds \left[ \thinspace V(U(t, x(t) + z)) - \Vs(t) \thinspace \right]_+  \right\} = 0.
\end{align*}
This verifies the claim above and finishes the proof of (\ref{est1}). The proof of (\ref{est2}) is similar; we just use the continuity of $V$ on $\mU\setminus \mU^0$ and the definition of  $\Vi$ for $U\in\mU^0$.

\vskip0.3cm
Finally, the sequence $x_n'$ converges to $x'$ in the sense of distributions. Also, the function $[\cdot]_+$  is convex. 
Therefore, integrating (\ref{est1}) and (\ref{est2}) on $[0,T]$ and passing to the limit, we obtain
$$
\int_0^T[\Vi(t)-x'(t)]_+\,dt=\int_0^T[x'(t)-\Vs(t)]_+\,dt=0.
$$
In particular, for almost every $t>0$ we have
$$
\Vi(t)\leq x'(t)\leq \Vs(t).
$$
\end{proof}

We end this section with an elegant formulation of the Rankine-Hugoniot condition and related entropy estimates, as originally presented by Dafermos in the $BV$ case. We show that the estimates remain true for solutions having the strong trace property (in fact, the strong trace property defined in \cite{Vasseur_trace} suffices). The proof is given in the appendix.
\begin{lemm}\label{Dafermos}
Consider $t\to x(t)$ a Lipschitzian path, and $U$ an entropic weak solution to (\ref{system}) verifying the strong trace property. Then, for almost every
$t>0$ we have
\begin{eqnarray*}
&&A(U(t,x(t)+))-A(U(t,x(t)-))=x'(t)(U(t,x(t)+)-U(t,x(t)-)),\\[0.1 cm]
&&G(U(t,x(t)+))-G(U(t,x(t)-))\leq x'(t)(\eta(U(t,x(t)+))-\eta(U(t,x(t)-))).
\end{eqnarray*}
Moreover, for almost every $t>0$ and $V\in\mV$
\begin{eqnarray*}
&&\frac{d}{dt}\int_{-\infty}^0\eta(U(t,y+x(t)) \midd V)\,dy\leq-F(U(t,x(t)-),V)+x'(t)\eta(U(t,x(t)-) \midd V),\\[0.2 cm]
&&\frac{d}{dt}\int_0^{\infty}\eta(U(t,y+x(t)) \midd V)\,dy\leq F(U(t,x(t)+),V)-x'(t)\eta(U(t,x(t)+) \midd V).
\end{eqnarray*}
\end{lemm}

\section{Proof of Theorem \ref{main}}

This section is dedicated to the proof of our main result, Theorem \ref{main}. We consider a system of conservation laws (\ref{system}) such that $A$ is $C^2$  on  a open, bounded, convex subset  $\mV$ of $\R^m$.
We assume that there exists  a strictly convex entropy $\eta$ of class $C^2$ on $\mV$ verifying (\ref{entropy flux}). 
We assume that $\eta$, $A$ and $G$ are continuous on $\mU=\overline{\mV}$. Let $(U_L,U_R)\in \mV^2$ be an entropic 1-discontinuity. For the neighborhood $B \owns U_L$ given by (H1) we define 
 $$
 \eps_0^2=\frac{1}{2}\inf_{U \notin B}\eta(U \midd U_L).
  $$  
Consider $U$ weak entropic solution of (\ref{system}) with values in $\mU$ on $(0,T)$ verifying the strong trace property of Definition \ref{defi_trace}, and 
\begin{eqnarray*}
&&\int_0^\infty \eta(U_0(x) \midd U_R)\,dx\leq \eps,\qquad
\int_{-\infty}^0\eta(U_0(x) \midd U_L)\,dx\leq \eps^4,
\end{eqnarray*}
for some $\eps<\eps_0$. Consider the path $t\to x(t)$ constructed in Proposition \ref{x(t)}. 
We first prove a pair of lemmas which provide estimates for the rate of change of the total relative entropy with respect to $x(t)$. 
\begin{lemm}\label{partie gauche}
For almost every time $0 < t <T$, we have 
$$
\frac{d}{dt}\int_{-\infty}^0\eta(U(t,y+x(t)) \midd U_L)\,dy\leq0.
$$
Furthermore, if we let $\mI\subset (0,T)$ denote the set of time such that $\eta(U(t,y+x(t)-) \midd U_L)\geq\eps^2$, then we have
$$
|\mI|\leq \eps.
$$
\end{lemm}
\begin{proof}
From Lemma \ref{Dafermos}, we have for almost every $t>0$ 
$$
\frac{d}{dt}\int_{-\infty}^0\eta(U(t,y+x(t)) \midd U_L)\,dy\leq-F(U(t,x(t)-),U_L)+x'(t)\eta(U(t,x(t)-) \midd U_L).
$$
At times $t$ for which $U(t,x(t)-)=U_L$, the right-hand side vanishes (and obviously $t\notin \mI$). For the other times $t$, Proposition \ref{x(t)} ensures that either
$$
 x'(t)\leq V(U(t,x(t)-))\leq\frac{F(U(t,x(t)-),U_L)}{\eta(U(t,x(t)-) \midd U_L)}-\eps, \qquad \ \ \mathrm{or} \ \ \qquad  x'(t)\leq V(U(t,x(t)+)).
$$
For the times corresponding to the first case, we have
$$
\frac{d}{dt}\int_{-\infty}^0\eta(U(t,y+x(t)) \midd U_L)\,dy\leq-\eps\eta(U(t,x(t)-) \midd U_L)\leq0.
$$
Especially, for every such  time in $\mI$, we have
\begin{equation}\label{sd}
\frac{d}{dt}\int_{-\infty}^0\eta(U(t,y+x(t)) \midd U_L)\,dy\leq-\eps^3.
\end{equation}
\\
For the second case, we appeal to Lemma 6. For almost every time we have either $U(t,x(t)-)=U(t,x(t)+)$ (which is covered by the first case) or  $x'(t)$ is the velocity associated to an entropic Rankine--Hugoniot discontinuity $(U(t,x(t)-),U(t,x(t)+))$. If $U(t,x(t)+)\in B\subset\mV$, we have
$$x'(t)\leq V(U(t,x(t)+)) \leq \lambda^-(U(t,x(t)+))-\eps,$$
which is in contradiction with Hypothesis (H2). Hence, for almost every time $t$ not covered by the first case, $U(t,x(t)+)\notin B$, and so $\eta(U(t,x(t)+) \midd U_L)\geq \eps_0^2$. In that case, we deduce from Lemmas \ref{decreasing} and \ref{Dafermos} that
\begin{align*}
\frac{d}{dt}\int_{-\infty}^0\eta(U(t,y+x(t)) \midd U_L)\,dy &\leq-F(U(t,x(t)-),U_L)+x'(t)\eta(U(t,x(t)-) \midd U_L) \\
& \leq -F(U(t,x(t)+),U_L)+x'(t)\eta(U(t,x(t)+) \midd U_L)\\
&\leq-\eps\eta(U(t,x(t)+) \midd U_L)\leq-\eps^3.
\end{align*}
In particular, for every $t\in\mI$, we have (\ref{sd}). Integrating in time, we find
$$
|\mI|\eps^3\leq \int_{-\infty}^0\eta(U_0(x) \midd U_L)\,dx\leq \eps^4,
$$
which gives the desired result.
\end{proof}
 
%

Before we prove the second lemma, we mention the following basic fact.
\vskip0.3cm
\noindent {\bf Remark.} If $f : \mU \to \R$ is $C^1$ on $\mV$ and bounded on $\mU$, then for any fixed $V \in \mV$, there exists $C_V >0$ such that for all $U \in \mU$
\begin{align*}
| f(U) - f(V) | \leq C_V | U -V |.
\end{align*}
The proof of this fact follows the same outline as the proof of Lemma \ref{L2}.

\begin{lemm}\label{partie droite}
There exists a constant $C>0$ such that for every $0<\eps<\eps_0$
$$
F(U(t,x(t)+),U_R)-x'(t)\eta(U(t,x(t)+) \midd U_R)\leq C\eps
$$
for almost every $t\notin \mI$.
\end{lemm}
\begin{proof}
First note that $t\notin\mI$ implies $U(t,x(t)-) \in B'= \{U\in B \ | \ \eta(U \midd U_L) \leq \eps^2_0 \} \subset B$. Also, from Lemma \ref{L2}, we have
$$
|U(t,x(t)-)-U_L|\leq C\eps.
$$
We consider three cases.
\vskip0.3cm
\begin{itemize}
\item[1.] (Points of continuity.)
For almost every $t\notin\mI$ such that 
$$
U(t,x(t)+)=U(t,x(t)-),
$$
we have, from  Proposition  \ref{x(t)} and Lemma \ref{lemmdisc},
$$
x'(t)\geq V(U(t,x(t)+)\geq\lambda^-(U_L)-C\eps.
$$
Hence, changing the constant $C$ from line to line if necessary,
\begin{eqnarray*}
&&F(U(t,x(t)+),U_R)-x'(t)\eta(U(t,x(t)+) \midd U_R)\\
&&\qquad\qquad \leq F(U(t,x(t)-),U_R)-\lambda^-(U_L)\eta(U(t,x(t)-) \midd U_R)+C\eps\\
&&\qquad\qquad \leq F(U_L,U_R)-\lambda^-(U_L)\eta(U_L \midd U_R)+C\eps\\
&&\qquad\qquad \leq F(S_{U_L}(0),S_{U_L}(s_0))-\sigma_{U_L}(0)\eta(S_{U_L}(0) \midd S_{U_L}(s_0))+C\eps\leq C\eps,
\end{eqnarray*}
where, in the last inequality, we used Lemma \ref{cornerstone} with $S_{U_L}(s_0) = U_R$. 
\item[2.] (Discontinuities excluding 1-shocks.)
If $(U(t,x(t)-),U(t,x(t)+))$ is an entropic Rankine--Hugoniot discontinuity with velocity $x'(t) \geq \lambda^-(U(t,x(t)-))$, then
\begin{eqnarray*}
&&F(U(t,x(t)+),U_R)-x'(t)\eta(U(t,x(t)+) \midd U_R)\\
&&\qquad\qquad \leq F(U(t,x(t)-),U_R)-x'(t)\eta(U(t,x(t)-) \midd U_R)\\
&&\qquad\qquad \leq F(U(t,x(t)-),U_R)-\lambda^-(U_-)\eta(U(t,x(t)-) \midd U_R)\\
&&\qquad\qquad \leq F(U_L,U_R)-\lambda^-(U_L)\eta(U_L \midd U_R)+C\eps\leq C\eps.
\end{eqnarray*}
In the second line we used Lemma \ref{decreasing}, and in the fourth line Lemma \ref{cornerstone}.
\item[3.] (1-shocks.)
On the other hand, if $(U(t,x(t)-),U(t,x(t)+))$ is an entropic discontinuity with velocity $x'(t) < \lambda^-(U(t,x(t)-))$, then, by Hypothesis (H3), it is a 1-shock and, since $U\to s_U$ is Lipschitz on $B$, there exists $s'\in [0,s_{U_L}]$ such that   we have both $U(t,x(t)+)=S_{U(t,x(t)-)}(s)=S_{U_L}(s')+ \mathcal{O}(\eps)$ and $x'(t)=\sigma_{U(t,x(t)-)}(s)=\sigma_{U_L}(s')+ \mathcal{O}(\eps)$.
As noted above, $U(t,x(t)-) \in B' \subsetneq B$.
Moreover $U\to s_U $ is continuous on $B'$, which is closed, and there is no state $U_+\in \mU^0$ such that $(U,U_+)$ is 
an admissible discontinuity. Hence the set  $\{S_{U}(s) \ | \ U\in B', s\in [0,s_U] \}$ is at a positive distance from $\mU^0$ (where $\eta$ and $A$ may be not $C^1$). The functions $F(\cdot, U_R)$ and $\eta(\cdot \midd U_R)$ are then uniformly Lipschitz on this set.
Hence, by Lemma \ref{cornerstone}, 
\begin{eqnarray*}
&&  F(U(t,x(t)+),U_R)-x'(t)\eta(U(t,x(t)+) \midd U_R)\\
&&\qquad\qquad \leq F(S_{U_L}(s'),U_R)-\sigma_{U_L}(s')\eta(S_{U_L}(s') \midd U_R)+ \mathcal{O}(\eps) \leq  \mathcal{O}(\eps).
\end{eqnarray*}
\end{itemize}
\end{proof}

We can now finish the proof of the theorem.
Lemma \ref{partie gauche} implies that for all $0<t<T$
\begin{equation}\label{1}
\int_{-\infty}^0\eta(U(t,y+x(t)) \midd U_L)\,dy\leq \int_{-\infty}^0\eta(U_0(y) \midd U_L)\,dy \leq \eps^4.
\end{equation}
On the other hand,
\begin{multline*}
\int_{0}^\infty\eta(U(t,y+x(t)) \midd U_R)\,dy\\
=\int_{0}^\infty\eta(U_0(y) \midd U_R)\,dy+\int_{(0,t)\cap\mI}\frac{d}{dr}\left\{\int_0^\infty\eta(U(r,y+x(r)) \midd U_R)\,dy\right\}\,dr\\
+\int_{(0,t)\cap\mI^c}\frac{d}{dr}\left\{\int_0^\infty\eta(U(r,y+x(r)) \midd U_R)\,dy\right\}\,dr.
\end{multline*}
Since $F$, $x'$, and $\eta$ are bounded and $|\mI|\leq \eps$, we have
\begin{multline*}
\left|\int_{\mI}\frac{d}{dr}\left\{\int_0^\infty\eta(U(r,y+x(r)) \midd U_R)\,dy\right\}\,dr\right|\\
=\int_{\mI}\left|F(U(r,x(r)+),U_R)-x'(r)\eta(U(r,x(r)+) \midd U_R)\right|\,dr\leq C|\mI|\leq C\eps.
\end{multline*}
Applying this estimate together with Lemmas \ref{L2} and \ref{partie droite} above, we obtain
\begin{equation}\label{2}
\int_{0}^\infty\eta(U(t,y+x(t)) \midd U_R)\,dy\leq C\eps(1+t).
\end{equation}
It remains to show that $|x(t)-\sigma t|\leq C\sqrt{\eps t(1+t)}.$ We denote by $\sU$ the function defined by 
\begin{align*}
\sU(x) = 
\begin{cases}
U_L, &\text{if $x<0$,}\\[0.1 cm]
U_R, &\text{if $x>0$.}\\
\end{cases}
\end{align*}
Note that the exact shock solution is $\sU(x-\sigma t)$.
By virtue of (\ref{1}) and (\ref{2}), we have for every $0<t<T$
$$
\|U(t,\cdot)-\sU(\cdot - x(t))\|_{L^2}\leq C\sqrt{(1+t)\eps}.
$$
for some $C > \sqrt{\frac{\eps_0^{3}}{C_1}}$, where $C_1$ is given by Lemma \ref{L2}. 
Let $M = \| x'(t) \|_{L^\infty}$. Then for $T>0$, we consider an even cutoff function $\phi\in C^\infty(\R)$ such that 
\begin{align*}
\begin{cases}
\phi(x)=1, &\text{if $|x|\leq MT$,}\\[0.1 cm]
\phi(x)=0, &\text{if $|x|\geq 2MT$,}\\[0.1 cm]
\phi '(x)\leq 0, &\text{if $x\geq0$,}\\[0.1 cm]
|\phi'(x)|\leq 2(MT)^{-1}, &\text{if $x\in \R$.}
\end{cases}
\end{align*}
Then, for almost every $0<t< T$, we have
\begin{align*}
0&=\int_0^t\int_{\R}\phi(x)\left[\dt U+\dx A(U)\right]\,dx\,dt\\
&=\int_{\R}\phi(x)\left[\sU(x-x(t))-\sU(x)\right]\,dx-\int_0^t\int_{\R}A(\sU(x-x(s)))\phi'(x)\,dx\,ds\\
& \qquad \qquad +\int_\R\phi(x)\left[U(t,x)-\sU(x-x(t))\right]\,dx+\int_\R\phi(x)\left[\sU(x)-U^0(x)\right]\,dx\\
& \qquad \qquad \qquad \qquad -\int_0^t\int_{\R}\left[A(U(s,x))-A(\sU(x-x(s)))\right]\phi'(x)\,dx\,ds.
\end{align*}
The terms on the second line above to reduce to
$$
x(t)(U_L-U_R)-t(A(U_L)-A(U_R))=(x(t)-\sigma t) (U_L-U_R).
$$ 
The third line can be controlled by
$$
\|\phi\|_{L^2(\R)}\left(\|U(t,\cdot)-\sU(\cdot-x(t))\|_{L^2(\R)}+\|U^0-\sU\|_{L^2(\R)}\right)\leq C\sqrt{4MT}\sqrt{(1+T)\eps}.
$$
Finally, since $A$ has a suitable Lipschitz property at the points $U_L, U_R \in \mV$ (see the remark preceding Lemma \ref{partie droite}), the last term has the following bound:
\begin{multline*}
\left|\int_0^t\int_\R \left[A(U(s,x))-A(\sU(x-x(s)))\right]\phi'(x)\,dx\,ds\right|\\
\qquad \ \ \leq C \|\phi'\|_{L^\infty(\R)}\int_0^t\int_{-2MT}^{2MT}\left|U(s,x)-\sU(x-x(s))\right|\,dx\,ds\\
\leq \frac{2C}{MT} \int_0^t \sqrt{4MT} \thinspace \| U(s,\cdot)-\sU(\cdot-x(s)) \|_{L^2(\R)} \,dx\,ds \leq C\sqrt{T(1+T)\eps}.
\end{multline*}
Combining the estimates above we obtain
$$
|x(t)-\sigma t|\leq \frac{C\sqrt{T (1+T)\eps}}{|U_L-U_R|}\leq C\sqrt{T(1+T)\eps}. 
$$
This concludes the proof of the theorem.


\section{Proof of the theorems stated in the introduction}

We can now prove Theorems \ref{theoEuler} and \ref{theoHyper} which were stated the introduction. In each case we want to show that the hypotheses of Theorem \ref{main} are fulfilled.

\subsection{Isentropic Euler system}

To illustrate the foregoing ideas, we apply our results to the isentropic Euler system (\ref{isentropic}). We will show that our hypotheses are satisfied for a fairly large class of pressure laws $P(\rho)$.

First, let us take
\begin{align*}
\mV = \left\{ (\rho, \rho u) \in \R^+ \times \R \ \big \vert \ 0 < \|(\rho,  u)\|_{L^{\infty}} < K \right\}.
\end{align*}
The petal-shaped region $\mV$ 
is open, bounded, and convex. In particular, $u$ is bounded in $\mV$, which ensures the continuity of $\eta$, $A$, and $G$ at the vacuum state $(0,0) \in \mU^0$. (Here, the set $\mU^0$ contains a single point.) One can check that differentiability fails at this point.
  
The eigenvalues for this system are
\begin{align}\label{ev1}
\lambda_1(\rho,\rho u) = u - \sqrt{P^{\prime}(\rho)}, \qquad \lambda_2(\rho,\rho u) = u + \sqrt{P^{\prime}(\rho)}.
\end{align}
To ensure hyperbolicity we assume $P^{\prime}(\rho) >0$. If $\left\{\sigma, (\rho_L, u_L), (\rho_R, u_R) \right\}$ 
satisfies the Rankine-Hugoniot condition, then the following relations hold
\begin{align}
u_L \pm \sqrt{\frac{\rho_R}{\rho_L} \left[ \frac{P_R - P_L}{\rho_R - \rho_L} \right]} \thinspace =  \thinspace \sigma \thinspace = \thinspace u_R \pm \sqrt{\frac{\rho_L}{\rho_R} \left[\frac{P_R - P_L}{\rho_R - \rho_L}\right]}, \label{rh1} \\[0.2 cm]
(u_L - u_R)^2 = \frac{(P_R - P_L)(\rho_R -\rho_L)}{\rho_R \rho_L}. \ \ \ \thinspace \ \qquad \label{rh2}
\end{align}
In view of (\ref{ev1}), the plus and minus signs in (\ref{rh1}) distinguish the 2-shocks ($\rho_R < \rho_L$) and 1-shocks ($\rho_R > \rho_L$), respectively. It will not be necessary to distinguish between shocks and contact discontinuities here, so we will call all such discontinuities shocks. Let us assume that the entropy admissible shocks are characterized by the relation $u_L > u_R$. The opposite case is treated similarly.

With respect to the hypotheses (H1) and ($\text{H}^{\prime}$1), we take
\begin{align*}
S_{(\rho, \rho u)}(s) &= \left( \rho + s, \ \rho u \pm \rho \sqrt{\frac{(P(\rho + s) - P(\rho))s}{\rho(\rho + s)}} \thinspace \right),\\[0.2 cm]
\sigma_{(\rho, \rho u)}(s)  &= u \pm  \sqrt{\frac{\rho+s}{\rho} \left[ \frac{P(\rho + s) - P(\rho)}{s}\right]},
\end{align*}
where $s>0$ represents the change in density, in absolute value, across the shock. The plus and minus signs correspond to the cases ($\text{H}^{\prime}$1) and (H1), respectively. Note that, in both cases, the functions $(s,U)\to S_U(s)$ and $(s,U)\to\sigma_U(s)$ are $C^1$ on 
$\R^+ \times \mV$. Also, it is straightforward to verify that the function $U \to s_U$, where
\begin{align*}
s_U = \sup \left\{ s_0 \in \R^+ \ \big \vert \ S_U(s) \in \mV \ \text{for} \ \text{all} \ s\in [0,s_0] \right\}, 
\end{align*}
is Lipschitz on $B \subset \mV$, provided the neighborhood $B$ (corresponding to a fixed left or right endstate) is small enough. In particular, the terminal value $s_U$ is finite and does not correspond to the vacuum state.

Using only the hyperbolicity assumption on $P$, we can verify Hypotheses (H1b) and ($\text{H}^{\prime}$1b). Indeed, the relative entropy in this case is given by
\begin{align*}
\eta((\rho, \rho u) \midd (\rho_0, \rho_0 u_0)) = \ds \frac{\rho}{2} (u - u_0)^2 + S(\rho \midd \rho_0).
\end{align*}
Therefore,
\begin{align*}
\frac{d}{ds} \eta((\rho, \rho u) \midd S_{(\rho, \rho u)}(s)) &= \frac{d}{ds} \left\{ \frac{\rho}{2} \left[ \frac{(P(\rho + s) - P(\rho))s}{\rho(\rho + s)} \right] + S(\rho \midd \rho+s) \right\}\\[0.1 cm]
&= \frac{1}{2(\rho+s)^2} \left[ s(\rho+s)P^{\prime}(\rho + s) + \rho (P(\rho + s) - P(\rho))\right] + sS^{\prime \prime}(\rho + s) > 0.
\end{align*}

Finally, observe that the global Liu-admissibility hypotheses, (H1a) and ($\text{H}^{\prime}$1a), are satisfied if and only if the function
\begin{align*}
\phi(s) = \frac{\rho+s}{\rho} \left[ \frac{P(\rho+s) - P(\rho)}{s} \right] 
\end{align*}
is nondecreasing. The derivative of $\phi$ can be expressed as follows.
\begin{align*}
\phi^{\prime}(s) = 
\begin{cases}
\ds \frac{1}{\rho s^2} \ds \int_{\rho}^{\rho+s} (q-\rho)  [qP(q)]^{\prime \prime} \,dq \quad &\text{if $s \ne 0$;}\\[0.6 cm]
\ds \frac{1}{2 \rho}  [\rho P(\rho)]^{\prime \prime} &\text{if $s = 0$.}\\
\end{cases}
\end{align*}
Therefore, the Liu condition holds if and only if
\begin{align*}
[\rho P(\rho)]^{\prime \prime} \geq 0.
\end{align*}
This may include, for instance, systems which fail to be genuinely nonlinear. 
On the other hand, it is interesting to note that the genuine nonlinearity condition, $[\rho P(\rho)]^{\prime \prime} \ne 0$, holds if and only if $\sigma^{\prime} \ne 0$; that is, if and only if the shock speed is strictly monotone along each shock curve. In fact, the equivalence of genuine nonlinearity and the global {\it strict} Liu-condition can be shown for a wide class of conservation laws, as suggested by Liu \cite{Liu2}. 
Finally, it follows from (\ref{rh1}) and the analysis above that the Lax admissibility conditions also hold globally when $[\rho P(\rho)]^{\prime \prime} \geq 0$. This verifies (H2)-(H3) and the corresponding dual conditions. Therefore, Theorem \ref{main} applies and the proof is complete.\\

\subsection{Full Euler system}
Now let us show that our theorem applies to the full Euler system for a polytropic gas. The idea is generally the same; however, our presentation is less explicit due to the complexity of the system. First, we introduce the state space
\begin{align*}
\mV = \left\{ (\rho, \rho u, \rho E) \in \R^+ \times \R \times \R^+ \ \big \vert \ 0 < \|(\rho,  u,  E)\|_{L^{\infty}} < K \right\}.
\end{align*}
As before, the region $\mV$ is open, bounded, and convex; and the functions $\eta$, $A$, and $G$ are continuous at the vacuum state $(0,0) \in \mU^0$. 

To simplify the remainder of the presentation, we work with the nonconservative variables $(\rho, u, e)$. The sound speed in the case of a polytropic pressure law (\ref{polytropic}) is given by 
\begin{align*}
c = \sqrt{\partial_\rho P + \rho^{-2}P \partial_e P } = \sqrt{\gamma(\gamma -1) e}.
\end{align*}
Hence the eigenvalues of the system are
\begin{align}\label{ev2}
\lambda_1(\rho, u, e) = u - \sqrt{\gamma(\gamma -1) e}, \qquad \lambda_2(\rho, u, e) = u, \qquad \lambda_3(\rho, u, e) = u + \sqrt{\gamma(\gamma -1) e}.
\end{align}
We consider a discontinuity $\left\{\sigma, (\rho_L, u_L, e_L), (\rho_R, u_R, e_R) \right\}$ verifying the Rankine-Hugoniot condition. Excluding contact discontinuities $(u_L = u_R)$, the following relations hold.
\begin{align}
u_L \pm \sqrt{\ds \frac{\gamma P_L}{\rho_L}} \cdot \sqrt{\ds \ds\frac{\gamma-1}{2 \gamma} + \frac{\gamma+1}{2 \gamma} \left[ \frac{P_R}{P_L} \right]} \thinspace =  \thinspace \sigma \thinspace = \thinspace u_R \pm \sqrt{\ds \frac{\gamma P_R}{\rho_R}} \cdot \sqrt{\ds \ds\frac{\gamma-1}{2 \gamma} + \frac{\gamma+1}{2 \gamma} \left[ \frac{P_L}{P_R} \right]}, \label{rh3} \\[0.3 cm]
\ds\frac{P_R}{P_L} = \ds \frac{ \left[\ds \frac{\gamma+1}{\gamma-1} \right] \ds \frac{\rho_R}{\rho_L} -1 }{ \left[ \ds \frac{\gamma+1}{\gamma-1} \right]  - \ds \frac{\rho_R}{\rho_L}}, \qquad \ \ \ \
\ds\frac{e_R}{e_L} = \ds\frac{P_R}{P_L} \cdot \ds\frac{\rho_L}{\rho_R},  \qquad \ \ \ \ \ds \frac{\gamma -1}{\gamma+1} < \ds\frac{\rho_R}{\rho_L} < \ds \frac{\gamma+1}{\gamma-1}. \qquad \label{rh4}
\end{align}
We refer the reader to \cite{Toro} for a proof of these relations. Note that the plus and minus signs in (\ref{rh3}) distinguish the 3-shocks and 1-shocks, respectively. Let us fix a left endstate $(\rho_L, u_L, e_L) \in (\R^+ \times \R \times \R^+)$. The entropy admissible shocks can be parametrized by any of the three quantities $P_R$, $\rho_R$, and $e_R$, which are related by (\ref{rh4}). The following conditions are equivalent and characterize the family of 1-shocks.
\begin{align*}
\qquad \qquad \qquad \qquad P_R > P_L, \qquad \rho_R > \rho_L, \qquad e_R > e_L. \qquad  \qquad  \text{(1-shock)}
\end{align*}
Similarly, the 3-shocks are characterized by 
\begin{align*}
\qquad \qquad \qquad \qquad P_R < P_L, \qquad \rho_R < \rho_L, \qquad e_R < e_L. \qquad \qquad \text{(3-shock)}
\end{align*}
Note carefully that, for $\rho_L$ fixed, $\rho_R$ is bounded above and bounded away from zero, whereas $P_R$ and $e_R$ range over all positive real numbers. Furthermore, each of the quantities is an increasing function of the others as indicated above. While we omit explicit formulas for $S_{(\rho, \rho u, \rho E)}(s)$ and $\sigma_{(\rho, \rho u, \rho E)}(s)$, the regularity properties of (H1) and ($\text{H}^{\prime}$1) are easily satisfied.

Note that (\ref{ev2})-(\ref{rh4}) easily imply that both shock families verify the Lax and Liu admissibility conditions globally. Therefore, it remains only to check that the relative entropy increases along the shock curves. We will prove that (H1b) holds for the family of 1-shocks. The proof of ($\text{H}^{\prime}$1b) for the family of 3-shocks is essentially the same. A somewhat tedious calculation shows that
\begin{align*}
\eta((\rho_L, \rho_L u_L, \rho_L E_L) \midd (\rho_R, \rho_R u_R, \rho_R E_R)) =
(\gamma-1) h(\rho_L \midd \rho_R) - \rho_L \text{ln}(e_L \midd e_R) + \ds \frac{\rho_L}{2e_R} (u_L - u_R)^2,
\end{align*}
where $h(x) = x \thinspace \text{ln} \thinspace x$. Observe first that $h(x)$ and $-\text{ln} \thinspace x$ are both strictly convex functions. Since  $\vert \rho_R - \rho_L \vert$ and $\vert e_R - e_L \vert$ are increasing along the shock curve, the first two terms in the formula above are also increasing. To deal with the last term, we recall (\ref{rh2}) which also holds for the full Euler system and which we rewrite as follows.
\begin{align*}
(u_L - u_R)^2 = \ds \frac{P_R}{\rho_R} \left[1 - \ds \frac{P_L}{P_R} \right] \left[ \ds \frac{\rho_R}{\rho_L} - 1\right].
\end{align*}
Hence we have
\begin{align*}
\frac{\rho_L}{2e_R} (u_L - u_R)^2 = \frac{(\gamma-1) \rho_L \rho_R}{2P_R}(u_L - u_R)^2 = \ds \frac{1}{2} (\gamma-1) \rho_L \left[1 - \ds \frac{P_L}{P_R} \right] \left[ \ds \frac{\rho_R}{\rho_L} - 1\right].
\end{align*}
Finally, we note that the bracketed terms are both positive and increasing along the 1-shock curves, hence the product is increasing. This establishes hypothesis (H1b), and the theorem follows.

\subsection{Strictly Hyperbolic system with small amplitude}
The proof of Theorem \ref{theoHyper} follows the same outline as the proof of Theorem \ref{main}. We simply need to check that our structural hypotheses are verified locally. Again, we restrict our attention to 1-shocks and 1-contact discontinuities. The argument for n-shocks and n-contact discontinuites is similar. 

By assumption, the state space $\Omega$ is open and $\lambda^-(V)$ is a simple eigenvalue for all $V \in \Omega$. Therefore, in a neighborhood of each state $V \in \Omega$, there exists a 1-shock curve $S_{V}(s)$, of class $C^1$, verifying 
\begin{align}\label{gn}
\sigma^{\prime}_{V}(0) = \frac{1}{2} \frac{d}{ds} \bigg \vert_{s=0} \lambda_1(S_{V}(s)).
\end{align}
If the 1-characteristic family is genuinely nonlinear, the parameterization of $S_V(s)$ is chosen so that the right-hand side of ({\ref{gn}}) is negative. For linearly degenerate fields, the right-hand side vanishes and $\sigma^{\prime}_V(s)$ is identically zero for all $s$ along the shock curve $S_V(s)$. In either case, a straightforward continuity argument shows that for each $V_0 \in \Omega$ there exists $K>0$ (sufficiently small) so that for all $V \in B_K(V_0)$, the 1-shock curve $S_V(s)$ exists in a neighborhood of $V$ which contains $B_K(V_0)$ and such that $\sigma^{\prime}_V(s) \leq 0$ for $S_V(s) \in B_K(V_0)$. This means that the Liu condition (H1a) is valid with respect to any fixed $U_L \in B_K(V_0)$. In terms of the notation of Section \ref{hypos}, we have taken $\mV = B = B_K(V_0)$.

Now, with respect to property (H1b), an easy calculation shows that 
\begin{align*}
\ds \frac{d}{ds} \thinspace \eta(U \midd S_U(s)) = S^{\prime}_U(s)^{T} \cdot D^2 \eta (S_U(s)) \cdot [S_U(s)-U].
\end{align*}
Since $D^2 \eta$ is positive definite, the previous quantity is nonnegative for $s \geq 0$ sufficiently small (since $S_U(s)-U$ and $S^{\prime}_U(s)^{T}$ point in nearly the same direction). Hence, (H1b) holds for all $U_L \in B_K(V_0)$, provided $K$ is small enough.

For genuinely nonlinear fields, it is well-known that the Lax and Liu entropy conditions are equivalent for weak shocks. Therefore, it suffices to check conditions (H2) and (H3) among all entropic Rankine-Hugoniot discontinuities $(U_-, U_+)$ excluding the family of 1-discontinuities. Keep in mind that, in each case, both $U_-$ and $U_+$ are contained in the closure of $B_K(V_0)$.
Since all discontinuities under consideration are weak, we have 
\begin{align*}
\sigma(U_-, U_+) = \lambda(U_{\pm}) + \mathcal{O}(\vert U_+ - U_- \vert),
\end{align*}
where $\lambda(U_{\pm})$ is some intermediate eigenvalue of $\nabla A (U_{\pm})$. Since $\nabla A$ is continuous and $\lambda_1 = \lambda^-$ is a simple eigenvalue, there exists $\delta > 0$ such that $\lambda^-(V) < \lambda(V) - \delta$ for all $V \in B_K(V_0)$. Then, for $K$ small enough, it follows that $\sigma(U_-, U_+) \geq \lambda^-(U_{\pm})$ for all (intermediate) Rankine-Hugoniot discontinuities with endstates $U_-, U_+ \in \overline{B_K(V_0)}$. This completes the proof. 

\appendix
\section{Proof of Lemmas \ref{L2} and \ref{Dafermos}}
We first give  the proof of Lemma \ref{L2}.
\begin{proof}
On $\mV^2$, we have
$$
\eta(U|V)=\int_0^1\int_0^1 (U-V)^T \cdot D^2 \eta (V+st(U-V)) \cdot (U-V) \thinspace t\,ds\,dt.
$$
Consider a convex, compact set $\tilde{\Omega}$ such that $\Omega \subset \tilde{\Omega}\subset \mV$ and $\overline{\tilde{\Omega}^c} \cap \Omega = \emptyset$, where $\tilde{\Omega}^c = \mV \setminus \tilde{\Omega}$.
Also, let $0<\Lambda^-_{\tilde{\Omega}} \leq \Lambda^+_{\tilde{\Omega}}<\infty$ denote, respectively, the smallest and largest eigenvalues of $D^2 \eta$ on $\tilde{\Omega}$. Then, for any $U\in \tilde{\Omega}$ and $V\in \Omega$, we have
 $$
\frac{\Lambda^-_{\tilde{\Omega}}}{2} \thinspace |U-V|^2\leq\eta(U \midd V)\leq \frac{\Lambda^+_{\tilde{\Omega}}}{2} \thinspace |U-V|^2.
 $$
 On the other hand, there exists a constant $C>0$ such that
 $$
\inf_{U\in \tilde{\Omega}^c, V\in \Omega}\eta(U \midd V)\geq\frac{1}{C}, \qquad \sup_{U\in \tilde{\Omega}^c, V\in \Omega}\eta(U \midd V)\leq C.
 $$
Indeed, $\eta(\cdot \midd \cdot)$ is continuous on the compact set $\overline{\tilde{\Omega}^c}\times \Omega$. So it attains its maximum and minimum. The maximum is bounded. The mininum is attained at some point $(U_0,V_0)$ with $U_0\neq V_0$, and so $\eta(U_0 \midd V_0)\neq0$. 
Hence the results holds true with
\begin{align*}
C_1=\min\left\{\frac{\Lambda^-_{\tilde{\Omega}}}{2}, \ \frac{1}{C\ds\sup_{U\in\tilde{\Omega}^c,V\in\Omega}|U-V|^2}\right\}, \qquad 
C_2=\max\left\{\frac{\Lambda^+_{\tilde{\Omega}}}{2}, \ \frac{C}{\ds\inf_{U\in\tilde{\Omega}^c,V\in\Omega}|U-V|^2}\right\}.
\end{align*}
\end{proof}
\vskip0.5cm
We now give the proof of Lemma \ref{Dafermos}.
\begin{proof}
Since $A$, $\eta$, and $G$ are continuous on $\mU$, the functions of $(t,x)$ $A(U)$, $\eta(U \midd U_L)$, $\eta(U \midd U_R)$,
$F(U,U_L)$, and $F(U,U_R)$ also verify the strong trace property of Definition \ref{defi_trace}. 
\vskip0.5cm
\noindent 
Consider the family of mollifyer functions
$$
\phi_\eps(y)=\frac{1}{\eps}\phi_1(\frac{y}{\eps}),
$$ 
where $\phi_1\in C^\infty(\R)$ is nonnegative, compactly supported in $(-1,0)$, and such that
 $\int \phi_1(y)\,dy=1$.
We define 
$$
\Phi_\eps(x)=\int_x^\infty\phi_\eps(y)\,dy.
$$
Note that $\Phi_\eps$ is bounded by 1. It is equal to 1 for $x<-\eps$ and equal to 0 for $x>0$.
Especially, it converges to $\mathrm{1}_{\{x<0\}}$ when $\eps$ goes to 0.
\vskip0.5cm
\noindent We consider also a test function in time only $\psi\in C^\infty(\R^+)$, $\psi\geq0$, compactly supported
in $(0,T)\subset (0,\infty)$.
\vskip0.5cm
\noindent
For any function $V$ verifying the strong trace property of Definition \ref{defi_trace},
we have
\begin{eqnarray*}
&&\qquad\qquad\int_0^\infty\int_\R \psi(t)\phi_\eps(y-x(t))V(t,y)\,dy\,dt\\
&&=\int_0^\infty\psi(t)V(t,x(t)-)\,dt +\int_0^\infty\int_\R \psi(t)\phi_\eps(y-x(t))[V(t,y)-V(t,x(t)-)]\,dy\,dt.
\end{eqnarray*}
The last term is smaller than
\begin{eqnarray*}
&&\qquad \|\psi\|_{L^\infty}\int_0^T\int_{0}^{1}\phi_1(y)|V(t,x(t)-\eps y)-V(t,x(t)-)|\,dy\,dt\\
&&\leq \|\psi\|_{L^\infty}\|\phi_1\|_{L^\infty}\sup_{y\in(0,\eps)}\int_0^T|V(t,x(t)-y)-V(t,x(t)-)|\,dt.
\end{eqnarray*}
This converges to 0 for $\eps$ going to 0, thanks to the strong trace property.
Hence
\begin{equation}\label{eq-}
\lim_{\eps\to0}\int_0^\infty\int_\R \psi(t)\phi_\eps(y-x(t))V(t,y)\,dy\,dt=\int_0^\infty\psi(t)V(t,x(t)-)\,dt.
\end{equation}
In the same way, we show that
\begin{equation}\label{eq+}
\lim_{\eps\to0}\int_0^\infty\int_\R \psi(t)\phi_\eps(x(t)-y)V(t,y)\,dy\,dt=\int_0^\infty\psi(t)V(t,x(t)+)\,dt.
\end{equation}
And so
\begin{equation}\label{eq-+}
\lim_{\eps\to0}\int_0^\infty\int_\R \psi(t)(\phi_\eps(y-x(t))-\phi_\eps(x(t)-y))V(t,y)\,dy\,dt=\int_0^\infty\psi(t)[V(t,x(t)-)-V(t,x(t)+)]\,dt.
\end{equation}
\vskip0.5cm
\noindent 
We consider the test function in time and space 
$$\psi(t)\Phi_\eps(x-x(t)).$$
For the equation of Lemma \ref{defi_F} with $V\in\mV$, we get:
\begin{eqnarray*}
&&-\int_0^\infty\int_{\R}\psi(t)\Phi_\eps'(x-x(t))[x'(t)\eta(U(t,x) \midd V)-F(U(t,x),V)]\,dx\,dt\\
&&\qquad\qquad\geq-\int_0^\infty\int_{\R}\psi'(t)\Phi_\eps(x-x(t))\eta(U(t,x) \midd V)\,dx\,dt.
\end{eqnarray*}
Passing to the limit and using (\ref{eq-}), we find
$$
-\int_0^\infty \psi'(t)\int_{-\infty}^0\eta(U(t,x(t)+x) \midd V)\,dx\,dt\leq \int_0^\infty\psi(t)[x'(t)\eta(U(t,x(t)-) \midd V)-F(U(t,x(t)-),V)]\,dt.
$$
This is the desired result in the sense of distribution.
\vskip0.5cm
\noindent 
We consider now the test function in time and space 
$$\psi(t)\Phi_\eps(x(t)-x).$$
For the equation of Lemma \ref{defi_F} with $V\in\mV$, we get in the same way, using (\ref{eq+}),
$$
-\int_0^\infty \psi'(t)\int_0^{\infty}\eta(U(t,x(t)+x) \midd V)\,dx\,dt\leq \int_0^\infty\psi(t)[-x'(t)\eta(U(t,x(t)+) \midd V)+F(U(t,x(t)+),V)]\,dt.
$$
\vskip0.5cm
\noindent
We take, now, as a test function 
$$\psi(t)[\Phi_\eps(x-x(t))+\Phi_\eps(x(t)-x)-1].$$
Note that this function converges to 0 in $L^1$ when $\eps$ converges to 0.
For equation (\ref{system}), this gives
\begin{eqnarray*}
&&\int_0^\infty\int_{\R}\psi(t)[\Phi_\eps'(x-x(t))-\Phi_\eps'(x(t)-x)][x'(t)U(t,x)-A(U(t,x))]\,dx\,dt\\
&&\qquad\qquad=\int_0^\infty\int_{\R}\psi'(t)[\Phi_\eps(x-x(t))+\Phi_\eps(x(t)-x)-1]U(t,x)\,dx\,dt.
\end{eqnarray*}
The right hand side term converges to 0. Thanks to (\ref{eq-+}), the left hand side term 
converges to 
$$
\int_0^\infty\psi(t)[x'(t)(U(t,x(t)+)-U(t,x(t)-))-(A(U(t,x(t)+))-A(U(t,x(t)-)))]\,dt=0.
$$
This provides the first equality in the sense of distribution. The second one can be proven the same way from (\ref{entropie}).
\end{proof}

\noindent {\bf Acknowledgment:} The second author was  partially supported by the NSF while completing this work.

\bibliography{LV_arXiv.bib}

\begin{thebibliography}{10}

\bibitem{Bardos_Levermore_Golse1}
C.~Bardos, F.~Golse, and C.~D. Levermore.
\newblock Fluid dynamic limits of kinetic equations. {I}. {F}ormal derivations.
\newblock {\em J. Statist. Phys.}, 63(1-2):323--344, 1991.

\bibitem{Bardos_Levermore_Golse2}
C.~Bardos, F.~Golse, and C.~D. Levermore.
\newblock Fluid dynamic limits of kinetic equations. {II}. {C}onvergence proofs
  for the {B}oltzmann equation.
\newblock {\em Comm. Pure Appl. Math.}, 46(5):667--753, 1993.

\bibitem{BTV}
F.~Berthelin, A.~E. Tzavaras, and A.~Vasseur.
\newblock From discrete velocity {B}oltzmann equations to gas dynamics before
  shocks.
\newblock {\em J. Stat. Phys.}, 135(1):153--173, 2009.

\bibitem{BV}
F.~Berthelin and A.~Vasseur.
\newblock From kinetic equations to multidimensional isentropic gas dynamics
  before shocks.
\newblock {\em SIAM J. Math. Anal.}, 36(6):1807--1835 (electronic), 2005.

\bibitem{Bressan}
A.~Bressan.
\newblock {\em Hyperbolic systems of conservation laws: the one-dimensional
  {C}auchy problem}.
\newblock Oxford University Press, Oxford, 2000.

\bibitem{Bressan3}
A.~Bressan and R.~Colombo.
\newblock Unique solutions of {$2\times 2$} conservation laws with large data.
\newblock {\em Indiana Univ. Math. J.}, 44(3):677--725, 1995.

\bibitem{Bressan1}
A.~Bressan, G.~Crasta, and B.~Piccoli.
\newblock Well-posedness of the {C}auchy problem for {$n\times n$} systems of
  conservation laws.
\newblock {\em Mem. Amer. Math. Soc.}, 146(694):viii+134, 2000.

\bibitem{Bressan2}
A.~Bressan, T.-P. Liu, and T.~Yang.
\newblock {$L^1$} stability estimates for {$n\times n$} conservation laws.
\newblock {\em Arch. Ration. Mech. Anal.}, 149(1):1--22, 1999.

\bibitem{Chen3}
G.-Q. Chen.
\newblock Convergence of the {L}ax-{F}riedrichs scheme for isentropic gas
  dynamics. {III}.
\newblock {\em Acta Math. Sci. (English Ed.)}, 6(1):75--120, 1986.

\bibitem{Frid2}
G.-Q. Chen and H.~Frid.
\newblock Divergence-measure fields and hyperbolic conservation laws.
\newblock {\em Arch. Ration. Mech. Anal.}, 147(2):89--118, 1999.

\bibitem{Frid1}
G.-Q. Chen and H.~Frid.
\newblock Large-time behavior of entropy solutions of conservation laws.
\newblock {\em J. Differential Equations}, 152(2):308--357, 1999.

\bibitem{Chen1}
G.-Q. Chen, H.~Frid, and Y.~Li.
\newblock Uniqueness and stability of {R}iemann solutions with large
  oscillation in gas dynamics.
\newblock {\em Comm. Math. Phys.}, 228(2):201--217, 2002.

\bibitem{Chen_Li}
G.-Q. Chen and Y.~Li.
\newblock Stability of {R}iemann solutions with large oscillation for the
  relativistic {E}uler equations.
\newblock {\em J. Differential Equations}, 202(2):332--353, 2004.

\bibitem{Chen_trace}
G.-Q. Chen and M.~Rascle.
\newblock Initial layers and uniqueness of weak entropy solutions to hyperbolic
  conservation laws.
\newblock {\em Arch. Ration. Mech. Anal.}, 153(3):205--220, 2000.

\bibitem{Dafermos4}
C.~Dafermos.
\newblock Entropy and the stability of classical solutions of hyperbolic
  systems of conservation laws.
\newblock In {\em Recent mathematical methods in nonlinear wave propagation
  ({M}ontecatini {T}erme, 1994)}, volume 1640 of {\em Lecture Notes in Math.},
  pages 48--69. Springer, Berlin, 1996.

\bibitem{Dafermos1}
C.~M. Dafermos.
\newblock The second law of thermodynamics and stability.
\newblock {\em Arch. Rational Mech. Anal.}, 70(2):167--179, 1979.

\bibitem{DeLellis}
C.~De~Lellis, F.~Otto, and M.~Westdickenberg.
\newblock Structure of entropy solutions for multi-dimensional scalar
  conservation laws.
\newblock {\em Arch. Ration. Mech. Anal.}, 170(2):137--184, 2003.

\bibitem{DiPerna}
R.~J. DiPerna.
\newblock Uniqueness of solutions to hyperbolic conservation laws.
\newblock {\em Indiana Univ. Math. J.}, 28(1):137--188, 1979.

\bibitem{Glimm}
J.~Glimm.
\newblock Solutions in the large for nonlinear hyperbolic systems of equations.
\newblock {\em Comm. Pure Appl. Math.}, 18:697--715, 1965.

\bibitem{SaintRaymond1}
F.~Golse and L.~Saint-Raymond.
\newblock The {N}avier-{S}tokes limit of the {B}oltzmann equation for bounded
  collision kernels.
\newblock {\em Invent. Math.}, 155(1):81--161, 2004.

\bibitem{Hajek}
O.~H{\'a}jek.
\newblock Discontinuous differential equations. {I}, {II}.
\newblock {\em J. Differential Equations}, 32(2):149--170, 171--185, 1979.

\bibitem{Kwon}
Y.-S. Kwon.
\newblock Strong traces for degenerate parabolic-hyperbolic equations.
\newblock {\em Discrete Contin. Dyn. Syst.}, 25(4):1275--1286, 2009.

\bibitem{KV}
Y.-S. Kwon and A.~Vasseur.
\newblock Strong traces for solutions to scalar conservation laws with general
  flux.
\newblock {\em Arch. Ration. Mech. Anal.}, 185(3):495--513, 2007.

\bibitem{Lax}
P.~Lax.
\newblock Shock waves and entropy.
\newblock In {\em Contributions to nonlinear functional analysis ({P}roc.
  {S}ympos., {M}ath. {R}es. {C}enter, {U}niv. {W}isconsin, {M}adison, {W}is.,
  1971)}, pages 603--634. Academic Press, New York, 1971.

\bibitem{Leger}
N.~Leger.
\newblock {$L\sp 2$} stability estimates for shock solutions of scalar
  conservation laws using the relative entropy method.
\newblock {\em Arch. Ration. Mech. Anal.}, 2010, to appear.

\bibitem{Lewicka}
M.~Lewicka and K.~Trivisa.
\newblock On the {$L^1$} well posedness of systems of conservation laws near
  solutions containing two large shocks.
\newblock {\em J. Differential Equations}, 179(1):133--177, 2002.

\bibitem{Lions_Masmoudi}
P.-L. Lions and N.~Masmoudi.
\newblock From the {B}oltzmann equations to the equations of incompressible
  fluid mechanics. {I}, {II}.
\newblock {\em Arch. Ration. Mech. Anal.}, 158(3):173--193, 195--211, 2001.

\bibitem{LPT}
P.-L. Lions, B.~Perthame, and E.~Tadmor.
\newblock Kinetic formulation of the isentropic gas dynamics and {$p$}-systems.
\newblock {\em Comm. Math. Phys.}, 163(2):415--431, 1994.

\bibitem{Liu2}
T.-P. Liu.
\newblock The {R}iemann problem for general systems of conservation laws.
\newblock {\em J. Differential Equations}, 18:218--234, 1975.

\bibitem{Ruggeri}
T.-P. Liu and T.~Ruggeri.
\newblock Entropy production and admissibility of shocks.
\newblock {\em Acta Math. Appl. Sin. Engl. Ser.}, 19(1):1--12, 2003.

\bibitem{liu}
T.-P. Liu and T.~Yang.
\newblock Well-posedness theory for hyperbolic conservation laws.
\newblock {\em Comm. Pure Appl. Math.}, 52(12):1553--1586, 1999.

\bibitem{SaintRaymond3}
N.~Masmoudi and L.~Saint-Raymond.
\newblock From the {B}oltzmann equation to the {S}tokes-{F}ourier system in a
  bounded domain.
\newblock {\em Comm. Pure Appl. Math.}, 56(9):1263--1293, 2003.

\bibitem{MV}
A.~Mellet and A.~Vasseur.
\newblock Asymptotic analysis for a {V}lasov-{F}okker-{P}lanck/compressible
  {N}avier-{S}tokes system of equations.
\newblock {\em Comm. Math. Phys.}, 281(3):573--596, 2008.

\bibitem{Panov2}
E.~Y. Panov.
\newblock Existence of strong traces for generalized solutions of
  multidimensional scalar conservation laws.
\newblock {\em J. Hyperbolic Differ. Equ.}, 2(4):885--908, 2005.

\bibitem{Panov}
E.~Y. Panov.
\newblock Existence of strong traces for quasi-solutions of multidimensional
  conservation laws.
\newblock {\em J. Hyperbolic Differ. Equ.}, 4(4):729--770, 2007.

\bibitem{SaintRaymond4}
L.~Saint-Raymond.
\newblock Convergence of solutions to the {B}oltzmann equation in the
  incompressible {E}uler limit.
\newblock {\em Arch. Ration. Mech. Anal.}, 166(1):47--80, 2003.

\bibitem{SaintRaymond2}
L.~Saint-Raymond.
\newblock From the {BGK} model to the {N}avier-{S}tokes equations.
\newblock {\em Ann. Sci. \'Ecole Norm. Sup. (4)}, 36(2):271--317, 2003.

\bibitem{Smoller}
J.~A. Smoller and J.~L. Johnson.
\newblock Global solutions for an extended class of hyperbolic systems of
  conservation laws.
\newblock {\em Arch. Rational Mech. Anal.}, 32:169--189, 1969.

\bibitem{Toro}
E.~F. Toro.
\newblock {\em Riemann solvers and numerical methods for fluid dynamics}.
\newblock Springer-Verlag, Berlin, second edition, 1999.
\newblock A practical introduction.

\bibitem{Tzavaras_theory}
A.~E. Tzavaras.
\newblock Relative entropy in hyperbolic relaxation.
\newblock {\em Commun. Math. Sci.}, 3(2):119--132, 2005.

\bibitem{Vasseur_gamma3}
A.~Vasseur.
\newblock Time regularity for the system of isentropic gas dynamics with
  {$\gamma=3$}.
\newblock {\em Comm. Partial Differential Equations}, 24(11-12):1987--1997,
  1999.

\bibitem{Vasseur_shock}
A.~Vasseur.
\newblock Existence and properties of semidiscrete shock profiles for the
  isentropic gas dynamic system with {$\gamma=3$}.
\newblock {\em SIAM J. Numer. Anal.}, 38(6):1886--1901 (electronic), 2001.

\bibitem{Vasseur_trace}
A.~Vasseur.
\newblock Strong traces for solutions of multidimensional scalar conservation
  laws.
\newblock {\em Arch. Ration. Mech. Anal.}, 160(3):181--193, 2001.

\bibitem{Vasseur_Book}
A.~Vasseur.
\newblock Recent results on hydrodynamic limits.
\newblock In {\em Handbook of differential equations: evolutionary equations.
  {V}ol. {IV}}, Handb. Differ. Equ., pages 323--376. Elsevier/North-Holland,
  Amsterdam, 2008.

\bibitem{Yau}
H.-T. Yau.
\newblock Relative entropy and hydrodynamics of {G}inzburg-{L}andau models.
\newblock {\em Lett. Math. Phys.}, 22(1):63--80, 1991.

\end{thebibliography}

\end{document}